\documentclass{article}
\usepackage{fullpage}

\usepackage[utf8]{inputenc}
\usepackage{amsthm,amssymb}
\usepackage{amsmath}
\usepackage{thmtools,mathtools}
\usepackage{caption,subcaption}
\usepackage{epsfig,graphicx,graphics,color,tikz,tikz-cd,tcolorbox,stackengine}
\usepackage{enumerate,enumitem}
\usepackage[sort]{cite}
\usepackage{xspace}
\usepackage{bbm}
\usepackage{hyperref}
\usepackage[capitalise]{cleveref}
\hypersetup{
    colorlinks=true,       % false: boxed links; true: colored links
    linkcolor=blue,        % color of internal links (change box color with linkbordercolor)
    citecolor=magenta,         % color of links to bibliography
    filecolor=magenta,     % color of file links
    urlcolor=cyan,         % color of external links
    linktoc=all
}
\usepackage{titling}

\theoremstyle{plain}
\newtheorem{thm}{Theorem}[section]
\newtheorem{lem}[thm]{Lemma}
\newtheorem{cor}[thm]{Corollary}
\newtheorem{cl}[thm]{Claim}
\newtheorem{prop}[thm]{Proposition}

\theoremstyle{definition}

\newtheorem{ex}[thm]{Example}

\newtheorem{rem}[thm]{Remark}

\newcommand*{\claimproofname}{Proof of claim.}
\newenvironment{claimproof}[1][\claimproofname]{\begin{proof}[#1]}{\end{proof}}

\def\final{0}  % set this to 1 to get a comment-free version
\def\iflong{\iffalse}
\ifnum\final=0  %namely if we allow comments in the output
\newcommand{\kristof}[1]{{\color{red}[{\textbf{Kristóf:} #1}]\marginpar{\color{red}*}}}
\newcommand{\tamas}[1]{{\color{blue}[{\textbf{Tamás:} #1}]\marginpar{\color{blue}*}}}
\newcommand{\andris}[1]{{\color{magenta}[{\textbf{Andris:} #1}]\marginpar{\color{magenta}*}}}
\newcommand{\bogi}[1]{{\color{green}[{\textbf{Bogi:} #1}]\marginpar{\color{green}*}}}
\newcommand{\laci}[1]{{\color{purple}[{\textbf{Laci:} #1}]\marginpar{\color{purple}*}}}
\else % in this case [final=1] we don't want any comments to show
\newcommand{\kristof}[1]{}
\newcommand{\tamas}[1]{}
\newcommand{\andris}[1]{}
\newcommand{\bogi}[1]{}
\newcommand{\laci}[1]{}
\fi

\DeclareMathOperator*{\argmax}{arg\,max}

\newcommand{\msf}{\text{\tt msf}}

\def\obt#1{\overbracket[.5pt][2pt]{#1}}
\def\ubt#1{\underbracket[.5pt][2pt]{#1}}
\newcommand{\R}{\mathbb{R}}

\newcommand{\bR}{\mathbb{R}}
\newcommand{\bQ}{\mathbb{Q}}
\newcommand{\bZ}{\mathbb{Z}}

\newcommand{\bE}{\mathbb{E}}

\newcommand{\cB}{\mathcal{B}}
\newcommand{\cC}{\mathcal{C}}
\newcommand{\cE}{\mathcal{E}}

\newcommand{\cK}{\mathcal{K}}

\newcommand{\cQ}{\mathcal{Q}}
\newcommand{\cT}{\mathcal{T}}

\newcommand{\cA}{\mathcal{A}}

\newcommand{\cU}{\mathcal{U}}

\let\Right\bigr
\let\Left\bigl
\makeatletter
\def\bigr#1{\Right#1\@ifnextchar){\!\bigr}{}}
\def\bigl#1{\Left#1\@ifnextchar({\!\bigl}{}}
\makeatother

\title{Monotonic Decompositions of Submodular Set Functions}

\thanksmarkseries{alph}
\author{
Kristóf Bérczi\thanks{MTA-ELTE Matroid Optimization Research Group and HUN-REN–ELTE Egerváry Research Group, Department of Operations Research, Eötvös Loránd University, and HUN-REN Alfréd Rényi Institute of Mathematics, Budapest, Hungary. Email: \texttt{kristof.berczi@ttk.elte.hu}.}
\and
Boglárka Gehér\thanks{Department of Applied Analysis and Computational Mathematics, Eötvös Loránd University, and HUN-REN Alfréd Rényi Institute of Mathematics, Budapest, Hungary. Email: \texttt{bogigeher@gmail.com}.}
\and
András Imolay\thanks{Department of Operations Research, Eötvös Loránd University, Budapest, Hungary. Email: \texttt{andras.imolay@ttk.elte.hu}.}
\and
László Lovász\thanks{HUN-REN Alfréd Rényi Institute of Mathematics, Budapest, Hungary. Email: \texttt{laszlo.lovasz@ttk.elte.hu}.}
\and
Tamás Schwarcz\thanks{MTA-ELTE Matroid Optimization Research Group, Department of Operations Research, Eötvös Loránd University, Budapest, Hungary. Email: \texttt{tamas.schwarcz@ttk.elte.hu}.}
}

\date{}

\begin{document}
\maketitle
\tableofcontents

%%%%%%%%%%%%%%%%%%%%%%%%%%%%%%%%
 \newpage
%%%%%%%%%%%%%%%%%%%%%%%%%%%%%%%%

%%%%%%%%%%%%%%%%%%%%%%%%%%%%%%%%
\begin{abstract} 
Submodular set functions are undoubtedly among the most important building blocks of combinatorial optimization. Somewhat surprisingly, continuous counterparts of such functions have also appeared in an analytic line of research where they found applications in the theory of finitely additive measures, nonlinear integrals, and electric capacities. Recently, a number of connections between these two branches have been established, and the aim of this paper is to generalize further results on submodular set functions on finite sets to the analytic setting.

We first extend the notion of duality of matroids to submodular set functions, and characterize the uniquely determined decomposition of a submodular set function into the sum of a nonnegaive charge and an increasing submodular set function in which the charge is maximal. Then, we describe basic properties of infinite-alternating set functions, a subclass of submodular set functions that serves as an analytic counterpart of coverage functions. By relaxing the monotonicity assumption in the definition, we introduce a new class of submodular functions with distinguished structural properties that includes, among others, weighted cut functions of graphs. We prove that, unlike general submodular set functions over an infinite domain, any infinite-alternating set function can be written as the sum of an increasing and a decreasing submodular function or as the difference of two increasing submodular functions, thus giving extension of results on monotonic decompositions in the finite case. Finally, motivated by its connections to graph parameters such as the maximum size of a cut and the maximum size of a fractional triangle packing, we study the structure of such decompositions for weighted cut functions of undirected graphs. 

\medskip


\end{abstract}
%%%%%%%%%%%%%%%%%%%%%%%%%%%%%%%%

%%%%%%%%%%%%%%%%%%%%%%%%%%%%%%%%
\section{Introduction}
\label{sec:intro}
%%%%%%%%%%%%%%%%%%%%%%%%%%%%%%%%

The study of submodular set functions goes back to the pioneering work of Rado~\cite{rado1942theorem} and Edmonds~\cite{edmonds2003submodular}. These functions are defined over subsets of a given set and capture the notion of diminishing returns, roughly meaning that the marginal gain from adding an element to a smaller set is at least as great as adding it to a larger set that contains the former. This property not only captures intuitive economic behaviors but also aligns with practical decision-making processes where resources are limited. As a result, submodular set functions have many practical applications, such as network design and facility location in combinatorial optimization~\cite{ageev19990,cornuejols1977location}, feature selection in machine learning and data mining~\cite{krause2005near}, auction design in economics~\cite{lehmann2006combinatorial}, and segmentation tasks in computer vision~\cite{boykov2001interactive,jegelka2011submodularity}. The main reason for the wide range of applications is that submodularity often allows for efficient optimization. An illustrative example is the greedy algorithm for maximizing a monotone submodular set function under a cardinality constraint~\cite{nemhauser1978analysis}, which provides a constant-factor approximation of the optimum value. We refer the interested reader to~\cite{frank2011connections,fujishige2005submodular,schrijver2003combinatorial} for a thorough introduction to the theory of submodular set functions, and to~\cite{bilmes2022submodularity} for a recent survey.

Less well known in the optimization community, submodularity has also appeared under the name ``strong subbaditivity'' in an analytic line of research that originated in the seminal paper of Choquet~\cite{choquet1954theory}. His work was inspired by a question raised by Brelot and Cartan concerning the relation between the interior and exterior Newtonian capacity of a Borel subset of $\bR^3$, and culminated in a comprehensive study of set functions defined on borelian sets. Recently, motivated by the growing demand for a limit theory of matroids, Lovász~\cite{lovasz2023submodular} described several connections between the combinatorial and analytic theory. 

One of Choquet's important constructions introduced in his paper is the Choquet integral of a bounded measurable function with respect to any monotone set function. This is extended in \cite{lovasz2023submodular} to integrating with respect to certain not necessarily monotone set functions, namely those with bounded variation. This latter property is equivalent to being able to decompose the set function as the sum of an increasing and of a decreasing (finite valued) set function. It is proved that every submodular set function can be written as such a sum. These two set functions are subadditive, which suggests the question whether one could require more: Can every bounded submodular set function be written as the sum of two submodular set functions, one increasing and one decreasing? One can also ask a related question: Can every bounded submodular set function be written as the difference of two submodular set functions, both increasing? 

Such decompositions play an important role in the analytic world. Besides defining integration with respect to them, another possible application is that in order to verify a statement for submodular functions, it is often sufficient to prove it for the monotone components separately --- a similar phenomenon can be observed in a number of cases concerning measures and signed measures. The problem is also motivated by combinatorial auctions and machine learning applications where the problem of minimizing the difference of submodular functions~\cite{narasimhan2005submodular,iyer2012algorithms} or a sum of a submodular and supermodular function~\cite{bai2018greed} arises naturally. 

%%%%%%%%%%%%%%%%
\subsection{Our contribution}
%%%%%%%%%%%%%%%%

In this paper, we generalize the notion of duality of matroids to submodular set functions, and use it to characterize the (uniquely determined) decomposition of a submodular set function into the sum of a nonnegaive charge and an increasing submodular set function in which the charge is maximal.

Next, we show that in the finite case, every submodular set function has a decomposition into the sum of an increasing and a decreasing submodular set function, but this fails in the infinite case. The situation with the question about difference will be similar. 

We also study the decomposability of submodular set functions into monotone parts under additional conditions. 
%We answer Questions~\ref{qu:main}\ref{it:aa} and~\ref{it:bb} in the negative.
While investigating the properties of capacities, Choquet~\cite{choquet1954theory} introduced the notion of infinite-alternating set functions, a class consisting of increasing submodular set functions with strong structural properties. He gave a thorough analysis of such functions from an analytical point of view, and determined the extreme rays of their convex cone. His work, however, has not yet been discussed in the context of discrete optimization. In addition to describing basic properties of infinite-alternating functions, we explain their relation to coverage functions, an important special case of submodular functions~\cite{bhaskar2019complexity,chakrabarty2012testing,chakrabarty2015recognizing}. We then initiate the study of a new class obtained by dropping the monotonicity assumption in Choquet's work, leading to the notion of {\it weakly infinite-alternating set functions.} We present several examples of infinite-alternating and weakly infinite-alternating set functions, and also show that any set function over a finite domain can be written as the difference of two infinite-alternating set functions. 

As a positive result for the infinite case, we show that the required decompositions always exist for weakly infinite-alternating functions, even in the stronger form when one of the functions is ought to be infinite-alternating and the other to be a finitely additive measure. One of the main technical ingredients is establishing a connection between the finite and infinite cases: The existence of a decomposition in the infinite case is often equivalent to the existence of a decomposition in the finite setting with absolute bounds on the function values. 

Weighted cut functions of undirected graphs are probably the most fundamental examples of weakly infinite-alternating set functions. Determining the maximum weight of a cut, called the \textsc{Max-Cut} problem, is a well studied graph theoretic problem that has received a lot of attention. It is one of Karp's 21 NP-complete problems~\cite{Karp1972}, and though it is known to be APX-hard, it has several applications in clustering, VLSI design and in network designs. This motivates the study of decompositions of weighted cut functions into monotone parts. We analyze the structure of such decompositions in more detail and reveal connections between the properties of the functions in the decomposition and certain graph parameters, such as the maximum size of a fractional packing of triangles.

\medskip

The rest of the paper is organized as follows. Basic definitions and notation are introduced in Section~\ref{sec:prelim}. In Section \ref{SEC:SBOUNDED} we construct decompositions under a strong boundedness condition. In Section \ref{sec:alternating}, we analyze the properties of alternating set functions and introduce the notion of weakly infinite-alternating set functions, a subclass of submodular set functions. Section~\ref{sec:main} presents constructions that disprove the existence of sum- and diff-decompositions in the infinite case: we first reduce the problems from infinite to finite domain in Section~\ref{sec:finite}, and then describe the counterexamples in Sections~\ref{sec:sum} and~\ref{sec:difference}, respectively. Section~\ref{sec:strongly} is devoted to the analysis of weakly infinite-alternating set functions, leading to an extension of sum-decompositions from the finite case to this class. Graph cut functions are discussed in Section~\ref{sec:cut} where properties of the decompositions are studied in the context of certain graph parameters. Several remarks and examples are included throughout to provide further insight and help the reader to understand the basic concepts.    

%%%%%%%%%%%%%%%%%%%%%%%%%%%%%%%%
\section{Preliminaries}
\label{sec:prelim}
%%%%%%%%%%%%%%%%%%%%%%%%%%%%%%%%

%%%%%%%%%%%%%%%%
\paragraph{Basic notation.}
%%%%%%%%%%%%%%%%

We denote the sets of \emph{reals}, \emph{rationals} and \emph{integers} by $\bR$, $\bQ$ and $\bZ$, respectively, and we add $+$ as a subscript when restricting the given set to \emph{nonnegative} values only. The set of \emph{even integers} is denoted by $2\cdot\bZ$. For a positive integer $k$, we use $[k]\coloneqq \{1,\dots,k\}$. If $X\subseteq J$ and $y\in J$, then $X\setminus \{y\}$ and $X\cup \{y\}$ are abbreviated as $X-y$ and $X+y$, respectively. Similarly, for a set function $f$ and $x\in J$, we simply write $f(x)$ instead of $f(\{x\})$. If $J$ is finite, $X\subseteq J$ and $w\colon J\to\bR$ is a weight function, then we use the notation $w(X)=\sum_{e\in X}w(e)$. For a vector $v\in\bR^n$, we denote by $v^+$ and $v^-$ the \emph{positive} and \emph{negative parts} of $v$, respectively, i.e., $(v^+)_i=\max\{0,v_i\}$ and $(v^-)_i=\max\{0,-v_i\}$ are nonnegative vectors with $v=v^+-v^-$.

%%%%%%%%%%%%%%%%
\paragraph{Submodular set functions.}
%%%%%%%%%%%%%%%%

Given a ground set $J$, a \emph{set algebra} $(J,\cB)$ is a family of subsets of $J$ that is closed under taking complements and finite unions. When $J$ is finite, we always assume that $\cB=2^J$. By a slight abuse of notation, we call the members of $\cB$ \emph{measurable}. A set function $\varphi \colon \cB \to \bR$ is \emph{increasing} if $X\subseteq Y$ implies $\varphi(X)\leq\varphi(Y)$, and \emph{decreasing} if it satisfies the reverse inequality. The function is \emph{bounded} if there exists $b\in\bR$ such that $-b\leq \varphi(X)\leq b$ for every $X\in\cB$. 
We set $\|\varphi\|=\sup\{|\varphi(X)|\mid X\in\cB\}$. The set function $\varphi$ is called \emph{submodular} if \[\varphi(X)+\varphi(Y) \geq \varphi(X \cap Y)+\varphi(X \cup Y)\] for all $X,Y \in \cB$. If $J$ is finite and $\cB=2^J$, then this is equivalent to $\varphi(X+u) + \varphi(X+v) - \varphi(X) - \varphi(X+u+v) \ge 0$ for all $X \subseteq V$, $u,v \in V\setminus X$ and $u \ne v$. We say that $\varphi$ is \emph{supermodular} if $-\varphi$ is submodular, and \emph{modular} if it is both sub- and supermodular. We say that the function is \emph{modular on a pair $A,B\in\cB$} if $\varphi(A)+\varphi(B)=\varphi(A\cap B)+\varphi(A\cup B)$. Note that shifting $\varphi$ by a constant preserves these properties hence we can usually assume that $\varphi(\emptyset)=0$. We call $\varphi$ \emph{normalized} if $\varphi(J)=1$. If $\cQ = \{Q_1,\dots, Q_q\}$ is a finite \emph{measurable partition} of $J$, that is, each $Q_i$ is measurable, then $\varphi/\cQ\colon 2^{[q]} \to \bR$ denotes the map defined by $(\varphi/\cQ)(X) = \varphi(\cup_{i \in X} Q_i)$ for $X \subseteq [q]$. We use $\cB_\cQ\coloneqq\{A\in\cB\mid A=\cup_{i\in I}Q_i\ \text{for some $I\subseteq [q]$}\}$. Adopting the terminology of \cite{rao1983theory}, we refer to a (not necessarily nonnegative) finitely additive measure as a \emph{charge}. When we say ``measure'', we tacitly mean a countably additive measure. So a measure is a countably additive nonnegative charge. We denote the \emph{Lebesgue measure} by $\lambda$.

%%%%%%%%%%%%%%%%
\paragraph{Monotonic decompositions.}
%%%%%%%%%%%%%%%%

Let $(J,\cB)$ be a set algebra and $\varphi \colon \cB \to \bR$ be a set function with $\varphi(\emptyset)=0$. A \emph{monotonic sum-decomposition} means writing up $\varphi$ as a sum $\varphi=\varphi_1+\varphi_2$ where $\varphi_1$ and $\varphi_2$ are increasing and decreasing submodular set functions, respectively, with $\varphi_1(\emptyset)=\varphi_2(\emptyset)=0$. A \emph{monotonic diff-decomposition} of $\varphi$ means writing up $\varphi$ as a difference $\varphi=\varphi_1-\varphi_2$ where both $\varphi_1$ and $\varphi_2$ are increasing submodular with $\varphi_1(\emptyset)=\varphi_2(\emptyset)=0$. By a \emph{monotonic decomposition}, we mean either a monotonic sum- or diff-decomposition. Note that $\varphi_1\geq \varphi$ necessarily holds in both cases. Using this terminology, the two questions posed in the introduction ask for the existence of monotonic sum- and diff-decompositions of any bounded submodular set function. For some $c\in\bR_+$, we call a monotonic decomposition \emph{$c$-bounded} if $\|\varphi_i\|\leq c\cdot\|\varphi\|$ for $i=1,2$. If $\varphi$ has finite domain then this is equivalent to 
\begin{equation*}
   \max_{\substack{X\in\cB\\ i\in[2]}}|\varphi_i(X)|\leq c\cdot\max_{X\in\cB}|\varphi(X)|.
\end{equation*}
We use $\|\varphi\|_+=\min\{\varphi_1(J)\mid \varphi=\varphi_1+\varphi_2\ \text{is a monotonic sum-decomposition}\}$, $\|\varphi\|_-=\min\{\varphi_1(J)\mid \varphi=\varphi_1-\varphi_2\ \text{is a monotonic diff-decomposition}\}$ and call a monotonic decomposition \emph{optimal} if $\varphi_1(J)$ attains the optimum value in the corresponding minimum.

As a warm-up, let us show why sum- and diff-decompositions exists in the finite case. For any sufficiently large $c>0$, the set function $\varphi_1(X)=\varphi(X)+c|X|$ is increasing submodular while $\varphi_2(X)=-c|X|$ is decreasing modular, hence $\varphi=\varphi_1+\varphi_2$ and $\varphi=\varphi_1-(-\varphi_2)$ provide sum- and diff-decompositions, respectively. In fact, we will see that diff-decompositions exist in a much stronger form: any set function $\varphi$ over a finite set can be written as the difference of two increasing submodular set functions that are infinite-alternating, see Theorem~\ref{thm:finite_diff-decomp}. Another decomposition, which generalizes to strongly bounded submodular set functions on infinite sets, will be described in Section \ref{SEC:SBOUNDED}. Let us note that this construction does not provide any bound in terms of the maximum value of $\varphi$ on the values of the functions appearing in the decomposition. 

As an example of a monotonic sum-decomposition in the infinite case, let $h\colon [0,1]\to\bR_+$ be a real-valued function with $h(0)=0$. It is easy to show (see e.g.~Denneberg \cite[Example 2.1]{Denn}) that the set function $\varphi(X)=h(\lambda(X))$ is submodular on the Borel subsets of $[0, 1]$ if and only if $h$ is concave. Assume now that $h$ is concave, and thus $\varphi$ is submodular. We claim that $\varphi$ admits a $1$-bounded monotonic sum-decomposition. If $h$ is increasing or decreasing, then the same holds for $\varphi$ and the problem is trivial. Otherwise, let $x_0=\argmax(h)$. We define $h_1(x)=h(x)$ if $x\leq x_0$ and $h_1(x)=h(x_0)$ otherwise, and $h_2(x)=0$ if $x\leq x_0$ and $h_2(x)=h(x)-h(x_0)$ otherwise. Note that $h_1$ and $h_2$ are concave, $h_1$ is increasing, $h_2$ is decreasing, and $h=h_1+h_2$. So $\varphi_1(X)=h_1(\lambda(X))$ and $\varphi_2(X)=h_2(\lambda(X))$ are submodular, $\varphi_1$ is increasing, $\varphi_2$ is decreasing, and $\varphi=\varphi_1+\varphi_2$. Since $\max\{\|\varphi_1\|,\|\varphi_2\|\}=\max\{h(x_0),|h(1)-h(x_0)|\}=h(x_0)=\|\varphi\|$ by the definition of $x_0$, the decomposition $\varphi=\varphi_1+\varphi_2$ is $1$-bounded. 

\paragraph{Hypergraph and graphs.}  Let $H=(V,\cE)$ be a hypergraph and $w\colon \cE\to\bR$ be a weight function. For a subset $\cE' \subseteq \cE$, we use the notation $w(\cE')\coloneqq \sum_{F \in \cE'} w(F)$. A hyperedge is said to \emph{connect} two disjoint subsets $X$ and $Y$ of $V$ if it intersects each of $X$ and $Y$. For two arbitrary subsets $X,Y\subseteq V$, we denote by $\cE[X,Y]$ the set of hyperedges connecting $X\setminus Y$ and $Y\setminus X$. The set function $d_{H, w}$ defined by $d_{H,w}(X) \coloneqq w(\cE[X, V\setminus X])$ is called the \emph{weighted cut function of $H$ with respect to $w$}.  
We denote by $\cE[X]$ \emph{the set of hyperedges induced by $X$}, that is, hyperedges $F \in \cE$ with $F \subseteq X$. When $H$ is a graph, i.e., each hyperedge has size $2$, we use similar notation with $E$ in place of $\cE$, and we denote the subgraph $(X, E[X])$ by $H[X]$. We use the notations $i_{H,w}(X) = w(\cE[X])$ and $e_{H,w}(X) = i_{H,w}(X) + d_{H,w}(X)$.
We dismiss the subscript $H$ whenever the hypergraph is clear from the context. It is not difficult to verify that 
\begin{align}
d_w(X)+d_w(Y)&=d_w(X\cap Y)+d_w(X\cup Y)+w(\cE[X,Y]\cap \cE[X \cup Y]) + w(\cE[X,Y]\cap \cE[V\setminus (X\cap Y)]), \label{eq:cut}\\
i_w(X)+i_w(Y)&=i_w(X\cap Y)+i_w(X\cup Y)-w(\cE[X,Y] \cap \cE[X \cup Y]),\ \text{and}\label{eq:i}\\    
e_w(X)+e_w(Y)&=e_w(X\cap Y)+e_w(X\cup Y)+ w(\cE[X,Y]\cap \cE[V\setminus (X\cap Y)])\label{eq:e}    
\end{align}
for all $X,Y\subseteq V$. Indeed, to prove any of the equalities, one has to check that every hyperedge of $H$ has the same contribution to both sides. In particular, if $w$ is nonnegative, then $d_w$ and $e_w$ are submodular while $i_w$ is supermodular. As a special case of hypergraphs, we consider undirected graphs $G=(V,E)$.

%%%%%%%%%%%%%%%%
\paragraph{Matroids.}
%%%%%%%%%%%%%%%%

For basic definitions on matroids, we refer the interested reader to~\cite{oxley2011matroid}. A \emph{matroid} $M=(J,r)$ is defined by its finite \emph{ground set} $J$ and its \emph{rank function} $r\colon2^J\to\bZ_+$ that satisfies the \emph{rank axioms}: (R1) $r(\emptyset)=0$, (R2) $X\subseteq Y\Rightarrow r(X)\leq r(Y)$, (R3) $r(X)\leq |X|$, and (R4) $r(X)+r(Y)\geq r(X\cap Y)+r(X\cup Y)$. Note that axiom (R4) asserts submodularity of the rank function. A set $X\subseteq J$ is called \emph{independent} if $r(X)=|X|$ and \emph{dependent} otherwise. An inclusionwise minimal dependent set is called a \emph{circuit}, and a \emph{loop} is a circuit of size $1$. Given pairwise disjoint sets $J_1\cup\dots\cup J_q\subseteq J$, the corresponding \emph{partition matroid} $M=(J,r)$ is defined by setting $r(X)=|\{i\mid X\cap J_i\neq\emptyset\}|$ for all $X\subseteq J$\footnote{Partition matroids could be defined in a more general form where $r(X)=\sum_{i=1}^q\min\{|X\cap J_i|,g_i\}$ for some $g_1,\dots,g_q\in\bZ_+$. However, in this paper, we are interested only in the case when $g_i=1$ for all $i\in[q]$ except for at most one which is $0$.}.

%%%%%%%%%%%%%%%%%%%%%%%%%%%%%%%%
\section{Strongly bounded set functions}
\label{SEC:SBOUNDED}
%%%%%%%%%%%%%%%%%%%%%%%%%%%%%%%%

In this section we describe two simple monotonic sum-decompositions, which work under a strong boundedness condition.

\subsection{Upper bounding charges}

Let $\varphi$ be a nonnegative submodular set function on a set algebra $(J,\cB)$ with $\varphi(\emptyset)=0$. We define $\obt\varphi\colon \cB\to\R\cup\{\infty\}$ by 
\begin{equation}\label{EQ:MU-SUP-FG}
\obt\varphi(X) = \sup \sum_{i=1}^n \varphi(X_i),
\end{equation}
where $\{X_1,\ldots,X_n\}$ ranges over all finite measurable partitions of $X$. We define $\obt\varphi(\emptyset)=0$. The set function $\obt\varphi$ as defined may take infinite values. We call the set function $\varphi\ge 0$ {\it strongly bounded} if $\obt\varphi$ is finite.

Let us recall some properties of strongly bounded set functions and the function $\obt{\varphi}$ from~\cite{lovasz2023submodular}. Considering the partition of $X$ into a single class, we see that $\obt\varphi\ge\varphi$. If $\varphi$ is a nonnegative charge, then $\obt\varphi=\varphi$; hence every nonnegative charge is strongly bounded. If $\varphi$ is a submodular set function, refining the partition $\{X_1,\dots,X_n\}$ does not decrease the value of $\sum_i \varphi(X_i)$. In particular, if the underlying set is finite, we have
\[
\obt\varphi(X)=\sum_{x\in X} \varphi(x).
\]
It is easy to see that both $\obt\varphi$ and $\obt\varphi-\varphi$ are increasing. If $\varphi\ge 0$ is a strongly bounded submodular set function on $(J,\cB)$ with $\varphi(\emptyset)=0$, then the set function $\obt\varphi$ is a nonnegative charge. Furthermore, if $\alpha\colon \cB\to\R_+$ is a nonnegative charge such that $\varphi\le\alpha$, then it follows easily that $\obt\varphi\le\alpha$. This implies that a nonnegative modular set function $\varphi$ with $\varphi(\emptyset)=0$ is strongly bounded if and only if there exists a nonnegative charge $\alpha$ on $(J,\cB)$ such that $\varphi\le\alpha$. If $\varphi$ is strongly bounded, then $\obt\varphi$ is the \emph{unique smallest nonnegative charge majorizing $\varphi$}. If, in addition to submodularity, $\varphi$ is continuous from below, then $\obt\varphi$ is a measure, implying that $\varphi$ is also continuous from above; see e.g.~\cite{lovasz2023submodular} for the precise definitions.

    To sum up, every strongly bounded submodular set function $\varphi$ with $\varphi=0$ has the decomposition
    \[
    \varphi = \obt\varphi + (\varphi-\obt\varphi)
    \]
    into the sum of a nonnegative charge and a decreasing submodular set function. Furthermore $\obt\varphi$ is the unique smallest nonnegative charge that can occur in such a decomposition.

\subsection{Duality and lower bounding charges}

Let us fix a charge $\eta$ on $(J,\cB)$, and let $\msf(\eta)$ denote the \emph{set of all increasing submodular set functions} $\varphi$ on $(J,\cB)$ with $\varphi(\emptyset)=0$ and $\varphi\le\eta$. Recall that the condition $\varphi\le\eta$ is equivalent to $\obt\varphi\le\eta$. For $\varphi\in\msf(\eta)$, we define its {\it dual with respect to $\eta$} as the set function
\[
\varphi^{\eta*}(X)= \varphi(J\setminus X) + \eta(X) - \varphi(J).
\]
This formula generalizes the formula for the rank function of the dual of a finite matroid, where $\eta$ is the counting measure. Note that $\eta\in\msf(\eta)$ and $\eta^{\eta*}=0$.

It is trivial that $\varphi^{\eta*}$ is submodular and $\varphi^{\eta*}(\emptyset)=0$. Furthermore, $\varphi^{\eta*}$ is increasing. The monotonicity of $\varphi$ implies that $\varphi^{\eta*}(X)\le\eta(X)$, so $\varphi^{\eta*}\in\msf(\eta)$. Hence $\eta$ is a charge majorizing $\varphi^{\eta*}$, which implies
\begin{equation}\label{EQ:ENVELOP}
\obt{\varphi^{\eta*}}\le\eta.
\end{equation}
For a given submodular set function $\varphi$, it is natural to consider $\eta=\obt\varphi$, and to define its {\it canonical dual} as
\begin{equation*}\label{EQ:CAN-DUAL1}
\varphi^*(X)= \varphi(J\setminus X) + \obt\varphi(X) - \varphi(J).
\end{equation*}
If $\varphi$ is a charge, then $\obt\varphi=\varphi$ and hence $\varphi^*=0$. So the canonical dual measures, in a sense, how non-modular $\varphi$ is. Also, $\varphi^{**}\coloneqq(\varphi^*)^*=0$, so $\varphi^{**}\not=\varphi$  as $\obt{\varphi^*}=0$. 

\begin{thm}
Let $(J,\cB)$ be a set algebra, and let $\varphi$ be an increasing submodular set function on $\cB$ with $\varphi(\emptyset)=0$. Then, $\ubt\varphi=\varphi-\varphi^{**}$ is a nonnegative charge. Furthermore, $\ubt\varphi$ is the unique largest charge $\alpha$ such that $\varphi-\alpha$ is increasing.
\end{thm}
\begin{proof}
In general, we have
\begin{align*}
\varphi^{**}(X) & =\varphi^*(J\setminus X)+\obt{\varphi^*}(X)-\varphi^*(J) \\
& = \varphi(X)+\obt\varphi(J\setminus X)-\varphi(J)+\obt{\varphi^*}(X)-\varphi^*(J) \\
& = \varphi(X) + \left(\obt\varphi(J)-\obt\varphi(X)\right)-\varphi(J) + \obt{\varphi^*}(X)-\left(\obt\varphi(J)-\varphi(J)\right) \\
&= \varphi(X) - \obt\varphi(X) + \obt{\varphi^*}(X)
\end{align*}
and hence
\[\varphi(X)-\varphi^{**}(X) = \obt\varphi(X)-\obt{\varphi^*}(X).\]
Recalling \eqref{EQ:ENVELOP}, we see that $\ubt\varphi=\varphi-\varphi^{**}$ is a nonnegative charge. For the second half, first observe that $\varphi-\ubt\varphi$ is increasing. Indeed, this follows from the monotonicity of $\varphi^{**}$. By the submodularity of $\varphi$, it is not difficult to see that $\varphi-\alpha$ is increasing for a charge $\alpha$ if and only if $\varphi(X)-\alpha(X)\le\varphi(J)-\alpha(J)$ holds for every $X\in\cB$. Let $\alpha$ be a charge satisfying this condition, then
\begin{align*}
\varphi^*(X) 
{}&{}= \varphi(J\setminus X)+\obt\varphi(X)-\varphi(J)\\
{}&{}\le \alpha(J\setminus X)+\varphi(J)-\alpha(J) +\obt\varphi(X) -\varphi(J)\\
{}&{}= \obt\varphi(X) -\alpha(X).
\end{align*}
So $\obt\varphi-\alpha$ is a charge majorizing $\varphi^*$, which implies that $\obt\varphi-\alpha\ge\obt{\varphi^*}$. By $\varphi(X)-\varphi^{**}(X) = \obt\varphi(X)-\obt{\varphi^*}(X)$, this is equivalent to $\alpha\le\ubt\varphi$.
\end{proof}

\begin{rem}\label{REM:STRONG-MAP}
In \cite{lovasz2023submodular}, a pair of submodular set functions $(\varphi,\psi)$ on the same set algebra is called {\it diverging}, if $\varphi-\psi$ is increasing. It was pointed out that this is a direct generalization of the matroid theoretical notion that the identity map $J\to J$ is a {\it strong map} $\varphi\to\psi$. So the identity map is strong both as $\obt\varphi\to\varphi$ and $\varphi\to\ubt\varphi$. Furthermore, $\obt\varphi$ is the unique smallest modular set function $\alpha$ on $J$ for which the identity map $\alpha\to\varphi$ is strong, and $\ubt\varphi$ is the unique largest modular set function $\beta$ on $J$ for which the identity map $\varphi\to\beta$ is strong.
\end{rem}

%%%%%%%%%%%%%%%%%%%%%%%%%%%%%%%%
\section{Alternating set functions}
\label{sec:alternating}
%%%%%%%%%%%%%%%%%%%%%%%%%%%%%%%%

The aim of this section is to give a better understanding of a highly structured but still rich class of submodular set functions. In Section~\ref{sec:kalternating}, we overview the notion of alternating set functions originally introduced by Choquet~\cite{choquet1954theory}, and prove some of their basic properties. Some of these results were already known before, but since they were mostly implicit in Choquet's work, we prove those here for the sake of completeness of the paper. Furthermore, we show that, in the finite case, these functions coincide with coverage functions, an important special case of submodular functions that arise in many applications. We then initiate the study of a weakening of the $k$-alternating property and establish a link between the two definitions in Section~\ref{sec:weaklykalternating}. To put the proposed class of functions into context, we describe some basic examples in Section~\ref{sec:hypergraph}. Finally, we show that any set function over a finite domain can be written as the difference of two infinite-alternating set functions.

%%%%%%%%%%%%%%%%%%%%%%%%%%%%%%%%
\subsection{\texorpdfstring{$k$}{k}-alternating set functions}
\label{sec:kalternating}
%%%%%%%%%%%%%%%%%%%%%%%%%%%%%%%%

Let $(J,\cB)$ be a set algebra and $\varphi\colon\cB\to\bR$ be a set function. For a positive integer $k$ and sets $A_0,A_1,\dots,A_k\in\cB$, we denote $V_\varphi(A_0;A_1,\dots,A_k)\coloneqq\sum_{K \subseteq [k]} (-1)^{|K|}\varphi\bigl(A_0 \cup \bigcup_{i \in K} A_i\bigr)$ and call the function \emph{$k$-alternating} if $\varphi(\emptyset)=0$ and
\begin{equation}
V_\varphi(A_0;A_1,\dots,A_k)\leq 0 \tag{\textsc{$k$-Alt}}\label{eq:kalt}
\end{equation}
for all $A_0, A_1, \ldots, A_k \in \cB$. If $\varphi$ is $k$-alternating for all positive integer $k$, then it is called \emph{infinite-alternating}.

\begin{lem}\label{lem:nonempty}
  We get an equivalent definition if we restrict $A_0$ to be disjoint from $\bigcup_{i=1}^k A_i$. Furthermore, it suffices to consider sets $A_j \neq \emptyset$ for all $j>0$. 
\end{lem}
\begin{proof}
For the first half of the lemma, observe that the left-hand side of \eqref{eq:kalt} does not change if we replace $A_i$ with $A_i \setminus A_0$ for $i=1,\dots,k$. For the second half, observe that $A_j=\emptyset$ implies that $A_0 \cup \bigcup_{i \in K} A_i=A_0 \cup \bigcup_{i \in K+j} A_i$ for all $K \subseteq [k]-j$, hence the inequality holds with equality as the terms of the sum cancel each other out pairwise.    
\end{proof}

For fixed sets $A_0, A_1, \ldots,  A_k$, the definition implies that $V_\varphi \coloneqq V_\varphi(A_0;A_1, \ldots , A_k)$ is a linear function of $\varphi$, i.e., $V_{\varphi+\psi}=V_\varphi+V_\psi$ and $V_{c \cdot \varphi}=c \cdot V_\varphi$ for arbitrary functions $\varphi$, $\psi$ and $c \in \bR$. Furthermore, we have 
\begin{equation}
    V_\varphi(A_0; A_1, \ldots , A_k)=V_\varphi(A_0; A_1, \ldots, A_{k-1})-V_\varphi(A_0 \cup A_k; A_1, \ldots , A_{k-1}),
    \label{eq:diminishing}
\end{equation}
i.e., the $k$-alternating sum can be expressed as the difference of two $(k-1)$-alternating sums. For $k=2$, inequality \eqref{eq:kalt} for pairwise disjoint $A_0, A_1, A_2$ is called the diminishing returns property of submodular set functions.

\begin{rem} \label{rem:discrete derivative}
    If $J$ is finite and $\varphi\colon 2^J \to \bR$ is a set function depending only on the cardinality, i.e., $\varphi(X)=g(|X|)$ for some function $g\colon\bZ_+ \to \bR$, then \eqref{eq:diminishing} can be considered as a generalization of discrete derivatives \cite{gleich2005finite}. For a function $f\colon \bZ_+ \to \bR$, let $\Delta f(n)\coloneqq \Delta^1 f(n)\coloneqq f(n+1)-f(n)$, and for $k\geq 2$ define $\Delta^k f$ recursively as $\Delta^k f\coloneqq \Delta(\Delta^{k-1} f)$. By induction, if $|A_0|=n$ and $a_1, a_2, \ldots , a_k$ are distinct elements of  $J \setminus A_0$, then equation \eqref{eq:diminishing} gives
    $V_\varphi(A_0; \{a_1\}, \ldots ,\{a_k\})=(-1)^k\Delta^k g(n)$.        
\end{rem}

The definition is monotone in the sense that a $k$-alternating set function is also $\ell$-alternating for any $1\leq \ell\leq k$. 

\begin{lem} \label{lem:alternating monotonicity}
    If $\varphi$ is $k$-alternating, then it is also $\ell$-alternating for all $\ell\in[k]$.
\end{lem}
\begin{proof}
    It suffices to verify the statement for $\ell=k-1$. Let $A_0,A_1,\dots,A_{k-1}\in\cB$ be arbitrary and set $A_k=A_{k-1}$. Then, we get
    \begin{align*}
    0
    {}&{}\geq
    V_\varphi(A_0; A_1, \ldots , A_k)\\
    {}&{}= 
    V_\varphi(A_0; A_1, \ldots , A_{k-2})-V_\varphi(A_0 \cup A_{k-1}; A_1, \ldots , A_{k-2})\\
    {}&{}-V_\varphi(A_0 \cup A_k; A_1, \ldots , A_{k-2})+V_\varphi(A_0 \cup A_{k-1}\cup A_k; A_1, \ldots , A_{k-2})
    \\
    {}&{}=
    V_\varphi(A_0; A_1, \ldots , A_{k-1})
    \end{align*}
    concluding the proof of the lemma.
\end{proof}

For $k=1$ and $k=2$, we get back the notion of increasing and increasing submodular set functions, respectively.

\begin{lem} \label{lem:12-alternating}
    A set function $\varphi$ is $1$-alternating if and only if it is increasing. A set function $\varphi$ is $2$-alternating if and only if it is increasing and submodular.
\end{lem}
\begin{proof}
    The first part is immediate from the definition. 
    
    For the second part, assume that $\varphi$ is 2-alternating. Then, by \cref{lem:alternating monotonicity}, it is 1-alternating and thus increasing as well. Let $X,Y\in\cB$ and define $A_0\coloneqq X\cap Y$, $A_1\coloneqq X\setminus Y$, $A_2\coloneqq Y\setminus X$. Then, \eqref{eq:kalt} gives
    \begin{align*}
        &\varphi(X\cap Y)-\varphi(X)-\varphi(Y)+\varphi(X\cup Y)\\
        ={}&{}
        \varphi(A_0)-\varphi(A_0 \cup A_1)-\varphi(A_0 \cup A_2)+\varphi(A_0 \cup A_1 \cup A_2)\\
        \leq{}&{} 0,
    \end{align*}
    implying submodularity.

    Conversely, assume that $\varphi$ is increasing and submodular and let $A_0, A_1, A_2\in\cB$. We can assume that $A_0 \cap A_1=A_0 \cap A_2=\emptyset$ by \cref{lem:nonempty}. Then, we get 
    \begin{align*}
    & \varphi(A_0)-\varphi(A_0 \cup A_1)-\varphi(A_0 \cup A_2)+\varphi(A_0 \cup A_1 \cup A_2)\\
    \leq{}&{} 
    \varphi(A_0 \cup (A_1 \cap A_2))-\varphi(A_0 \cup A_1)-\varphi(A_0 \cup A_2)+\varphi(A_0 \cup A_1 \cup A_2)\\
    \leq{}&{} 0,
    \end{align*}
   where the first inequality holds since $\varphi$ is increasing, while the second inequality holds by the submodularity of $\varphi$ when applied to the sets $A_0 \cup A_1$ and $A_0 \cup A_2$.
\end{proof}

There is a long line of research focusing on the characterization of extreme rays of the convex cone of submodular set functions, see~\cite{studeny2016core} for an overview. For infinite-alternating set functions, there is a simple characterization which relies on the following construction. For a set algebra $(J,\cB)$ and for any $A \in \cB$, let $\varphi_A \colon \cB \to \bR$ be defined as
\begin{equation} 
   \varphi_A(X)=\begin{cases}
    0 &\text{if $X \cap A=\emptyset$,} \\
    1 &\text{if $X \cap A\neq\emptyset$.}
    \end{cases}
    \label{eq:extremal_def}
\end{equation}
It is not difficult to see that $\varphi_A$ is a normalized infinite-alternating set function. Choquet~\cite[Section 43]{choquet1954theory} showed that in the finite case these set functions correspond to the extreme rays of the convex cone of infinite-alternating set functions, while a different proof appeared in~\cite{lovasz2023submodular}.

\begin{prop}[Choquet]\label{prop:extremal}
    The set of normalized extremal elements of the convex cone of infinite-alternating set functions over a finite set $J$ is $\{\varphi_A\mid \emptyset\neq A\subseteq J\}$.
\end{prop}

In the finite setting, Proposition~\ref{prop:extremal} provides an interesting connection between three combinatorial objects: infinite-alternating set functions, coverage functions, and rank functions of partition matroids. Assume that $J$ is finite and let $\varphi\colon 2^J\to\bR$ be a set function. Let $G=(J,2^J;E)$ be a bipartite graph where one vertex class corresponds to the elements of $J$, the other class corresponds to subsets of $J$, and there is an edge between $a\in J$ and $A\in 2^J$ if and only if $a\in A$. By Proposition~\ref{prop:extremal}, if $\varphi$ is infinite-alternating, then there exist nonnegative weights $\alpha_A\in\bR_+$ for $A\subseteq J$ such that $\varphi(X)=\sum[\alpha_A\mid A\in N(X)]$ for all $X\subseteq J$, where $N(X)$ denotes the set of neighbours of $X$ in $G$. Vice versa, it is not difficult to check that if $G=(J,U;E)$ is a bipartite graph and $\alpha_u\in\bR_+$ for $u\in U$, then $\varphi(X)\coloneqq \sum[\alpha_u\mid u\in N(X)]$ is an infinite-alternating set function over $J$. The latter construction may be familiar to those interested in combinatorial optimization. A \emph{coverage function} $f\colon 2^J\to\bR_+$ is determined by a universe $U$ of elements, weights $\alpha_u\in\bR_+$ for $u\in U$ and a family $\cU=\{U_a\mid a\in J\}$ of subsets of $U$ by setting $f(X)=\sum[\alpha_u\mid u\in \bigcup_{a\in X}U_a]$ for $X\subseteq J$. That is, $f$ is the infinite-alternating function corresponding to the bipartite graph $G=(J,U;E)$ with weights $\alpha_u\in\bR_+$ for $u\in U$, where $au$ is an edge if $u\in U_a$. Thus infinite-alternating functions over a finite domain are exactly the coverage functions. Finally, observe that $\varphi_A$ defined as in \eqref{eq:extremal_def} is the same as the rank function of the partition matroid with classes $A$ and $J\setminus A$ and bounds $1$ and $0$, respectively. Furthermore, the rank function of any partition matroid can be obtained as the sum of such functions. Therefore, by Proposition~\ref{prop:extremal}, the class of infinite-alternating set functions coincides with the convex cone of partition matroid rank functions.

\begin{rem} \label{rem:extremal_unique}
    It is not hard to prove that the functions $\{\varphi_A\mid \emptyset\neq A\subseteq J\}$ are linearly independent, see for example \cite{lovasz2023submodular} or the first proof of Theorem~\ref{thm:finite_diff-decomp}. Consequently, for any infinite-alternating function $\varphi \colon 2^J \to \bR$ there exists a unique decomposition $\varphi=\sum_{\emptyset \neq A \subseteq J} \alpha_A\varphi_A$ with $\alpha_A \in \bR_+$ for all $\emptyset \neq A \subseteq J$.
%\end{rem}
%\begin{rem}
    There are multiple ways to generalize \cref{prop:extremal} to infinite domains, but generally not all normalized extremal elements take the form $\varphi_A$ for some set $A \in \cB$, see~\cite{choquet1954theory,lovasz2023submodular}. 
\end{rem}

The following observations will be useful later on.

\begin{lem} \label{lem:extremal 1}
    Let $(J, \cB)$ be a set algebra and $A, A_0, A_1, \ldots , A_k \in \cB$. Then, 
    \begin{equation*}
        V_{\varphi_A}(A_0;A_1, \ldots , A_k)=\begin{cases}
        -1 &\text{if $A_0 \cap A=\emptyset$ and $A_i \cap A \neq \emptyset$ for all $1 \leq i \leq k$,} \\
        0 &\text{otherwise.}
    \end{cases}
    \end{equation*}
\end{lem}
\begin{proof}
    If $A_0 \cap A \neq \emptyset$, then $\varphi_A(A_0 \cup \bigcup_{i \in K} A_i)=1$ for all $K \subseteq [k]$, hence $V_{\varphi_A}(A_0;A_1, \ldots , A_k)=0$.

    If there exists an index $j\in[k]$ such that $A \cap A_j =\emptyset$, then $\varphi_A(A_0 \cup \bigcup_{i \in K} A_i)$ and $\varphi_A(A_0 \cup \bigcup_{i \in K+j} A_i)$ cancel out each other in the sum for all $K \subseteq [k]-j$, hence $V_{\varphi_A}(A_0; A_1, \ldots ,A_k)=0$.

    Finally, if $A_0 \cap A=\emptyset$ and $A_i \cap A \neq \emptyset$ for all $1 \leq i \leq k$, then $\varphi_A(A_0 \cup \bigcup_{i \in K} A_i)=1$ for all nonempty $K \subseteq [k]$, and $\varphi_A(A_0)=0$, hence $V_{\varphi_A}(A_0; A_1, \ldots ,A_k)=-1$.
\end{proof}

\begin{lem} \label{lem:k_0}
   Let $\varphi=\sum_{\emptyset \neq A \subseteq J} \alpha_A\varphi_A$ be an infinite-alternating set function over a finite domain written as the nonnegative combination of normalized extremal elements, and let $k_0$ be a positive integer. Then,  $V_\varphi(A_0; A_1, \ldots ,A_k)=0$ for all $k \geq k_0$ and pairwise disjoint sets $A_0, A_1, \ldots , A_k$ if and only if $\alpha_A=0$ for all $A \subseteq J$ with $|A| \geq k_0$.  
\end{lem}
\begin{proof}
    By linearity, $V_\varphi(A_0; A_1, \ldots ,A_k)=\sum_{\emptyset\neq A\subseteq J} \alpha_A V_{\varphi_A}(A_0; A_1, \ldots ,A_k)$ where $V_{\varphi_A}(A_0; A_1, \ldots ,A_k)\leq 0$ for all $\emptyset\neq A\subseteq J$. Therefore, it suffices to show that for a fix $A \subseteq J$, if $k >|A|$ then $V_{\varphi_A}(A_0; A_1, \ldots , A_k)=0$ for all pairwise disjoint sets $A_0, A_1, \ldots , A_k$, and if $k \leq |A|$ then there exist pairwise disjoint sets $A_0, A_1, \ldots , A_k$ with $V_{\varphi_A}(A_0; A_1, \ldots , A_k) < 0$. 

    If $k>|A|$, then for any pairwise disjoint sets $A_0, A_1, \ldots , A_k$, there exists an index $j\in[k]$ such that $A \cap A_j =\emptyset$, hence $V_{\varphi_A}(A_0; A_1, \ldots ,A_k)=0$ by \cref{lem:extremal 1}. 
    If $k \leq |A|$, let $A_0=\emptyset$ and $A_1, A_2, \ldots , A_k$ be pairwise disjoint subsets of $A$. By \cref{lem:extremal 1}, we get $V_{\varphi_A}(A_0; A_1, \ldots , A_k)=-1$. 
\end{proof}

\begin{lem} \label{lem:extremal 2}
     Let $J$ be a finite set, $\ell$ a positive integer, and $\varphi\coloneqq\sum [\varphi_A\mid \emptyset \neq A \subseteq J,|A| \leq \ell]$. Then, for any pairwise disjoint sets $A_0, A_1, \ldots , A_k$, $V_\varphi(A_0; A_1, \ldots , A_k)=0$ if and only if $k > \ell$, otherwise $V_\varphi(A_0; A_1, \ldots , A_k)<0$.
\end{lem}
\begin{proof}
    The `if' direction follows from \cref{lem:k_0}. For the `only if' direction, let $A_0, A_1, \ldots , A_k$ be pairwise disjoint sets with $k \leq \ell$. Choose $A$ such that $|A| \leq \ell$ and $A \cap A_i \neq \emptyset$ precisely if $1 \leq i \leq k$. By \cref{lem:extremal 1}, $V_{\varphi_A}(A_0;A_1, \ldots , A_k) =-1$, hence $V_\varphi(A_0;A_1, \ldots , A_k)<0$.
\end{proof}

%%%%%%%%%%%%%%%%%%%%%%%%%%%%%%%%
\subsection{Weakly \texorpdfstring{$k$}{k}-alternating set functions}
\label{sec:weaklykalternating}
%%%%%%%%%%%%%%%%%%%%%%%%%%%%%%%%

In certain cases, it will be more convenient to work with set functions satisfying the $k$-alternating property for pairwise disjoint sets, which motivates the following definition. Let $(J,\cB)$ be a set algebra. For a positive integer $k$, a set function $\varphi \colon \cB \to \bR$ is \emph{weakly $k$-alternating} if $\varphi(\emptyset)=0$ and it satisfies the inequality~\eqref{eq:kalt} for all pairwise disjoint $A_0, A_1, \ldots, A_k \in\cB$. The connection between the $k$-alternating and weakly $k$-alternating properties is given by the following lemma.

\begin{lem} \label{lem:weak_strong_alternating_connection}
    A set function $\varphi$ is $k$-alternating if and only if it is weakly $\ell$-alternating for all $\ell\in[k]$.
\end{lem}
\begin{proof}
    If $\varphi$ is $k$-alternating then, by \cref{lem:alternating monotonicity}, it is also $\ell$-alternating and so weakly $\ell$-alternating for $\ell\in [k]$.

    The other direction is a bit more technical and is proved using strong induction on $k$, and subject to this, by induction on $|\bigcup_{i=1}^k A_i |$. For $k=1$ and $k=2$, the statement is true by \cref{lem:12-alternating}. Therefore, assume that $k\geq 3$ and that the statement is true for all $\ell <k$ and for all sets $A_0', A_1', \ldots , A_k'$ with $|\bigcup_{i=1}^k A_i' |<|\bigcup_{i=1}^k A_i |$. By Lemma~\ref{lem:nonempty}, we may assume that $A_0 \cap (\bigcup_{i=1}^k A_i)=\emptyset$ and $A_i \neq \emptyset$ for all $i\in[k]$. 
    If $A_1, A_2, \ldots , A_k$ are pairwise disjoint, then the inequality follows by the weakly $k$-alternating property. Otherwise, by possibly relabeling the sets, there exists a $2 \leq \ell \leq k$ and a set $X$ with $\bigcap_{i=1}^\ell A_i=X$ and $X \cap \bigcup_{i=\ell+1}^k A_i=\emptyset$. By the inductive hypothesis, \eqref{eq:kalt} holds for the sets $A_0'=A_0 \cup X$ and $A_i'=A_i \setminus X$ for $i\in[k]$, hence it is enough to prove that
    \begin{equation*}
    V_\varphi(A_0;A_1, \ldots , A_k) \leq V_\varphi(A_0';A_1', \ldots , A_k').
    \end{equation*}
    The sets $A_0 \cup \bigcup_{i \in K} A_i$ and $A_0' \cup \bigcup_{i \in K} A_i'$ are equal if $K \cap [\ell] \neq \emptyset$, and otherwise their difference is exactly $X$. Thus, we get 
    \begin{align*}
    V_\varphi(A_0;A_1, \ldots , A_k)-V_\varphi(A_0';A_1', \ldots , A_k')
    = {}&{} \sum_{K \subseteq [k] \setminus [\ell]} (-1)^{|K|}\Bigl(\varphi\bigl(A_0 \cup \bigcup_{i \in K} A_i\bigr)-\varphi\bigl(X \cup A_0 \cup \bigcup_{i \in K} A_i\bigr)\Bigr) \\
    = {}&{} V_\varphi(A_0; A_{\ell+1}, A_{\ell+2}, \ldots , A_k, X)\\
    \leq {}&{}  0,
    \end{align*}
    where the last inequality follows by induction for sets $A_0, A_{\ell+1}, A_{\ell+2}, \ldots , A_k, X$. This finishes the proof of the lemma.
\end{proof}

The lemma implies the following characterization of infinite-alternating set functions.

\begin{cor}
    A set function $\varphi$ is infinite-alternating if and only if it is weakly $k$-alternating for all $k \geq 1$.
\end{cor}

It is a natural question whether there exists a $k$ such that every infinite-alternating function is already $k$-alternating. The following theorem answers this question in the negative.

\begin{thm}
    For any $\ell \geq 1$, there exists a function $\varphi$ which is $\ell$-alternating but not $(\ell+1)$-alternating.
\end{thm}

\begin{proof}
    Let $J$ be a finite set with $|J|>\ell$, and $X \subseteq J$ with $|X|=\ell+1$. Define
    \begin{equation*}
        \varphi\coloneqq\sum_{\substack{\emptyset \neq A \subseteq J \\ |A| \leq \ell}} \varphi_A,
     \end{equation*}
     and $\psi=\varphi-\varphi_X$.
     For $k \leq \ell$, using \cref{lem:extremal 1} on $\varphi_X$, \cref{lem:extremal 2} on $\varphi$, and linearity, 
     \begin{align*}
         V_\psi(A_0;A_1, \ldots ,A_k)= V_\varphi(A_0;A_1, \ldots ,A_k)-V_{\varphi_X}(A_0;A_1, \ldots ,A_k) 
         \leq 0
     \end{align*}
     for any pairwise disjoint sets $A_0, A_1, \ldots , A_k$, hence $\psi$ is $\ell$-alternating by \cref{lem:weak_strong_alternating_connection}.

     If $k=\ell+1$, $A_0=\emptyset$ and $A_i=\{x_i\}$ for $1 \leq i \leq k$, where $X=\{x_1,x_2, \ldots , x_k\}$, then again, using \cref{lem:extremal 1} on $\varphi_X$, \cref{lem:extremal 2} on $\varphi$, and linearity, 
     \begin{align*}
         V_\psi(A_0;A_1, \ldots ,A_k)= V_\varphi(A_0;A_1, \ldots ,A_k)-V_{\varphi_X}(A_0;A_1, \ldots ,A_k) = 1,
     \end{align*} 
     hence $\psi$ is not $(\ell+1)$-alternating.
\end{proof}

As a relaxation of infinite-alternating set functions, we introduce one of the key concepts of the paper and call a set function $\varphi \colon \cB \to \bR$ \emph{weakly infinite-alternating} if it is weakly $k$-alternating for all $k \geq 2$. At this point, it is probably not clear to the reader what exactly the relationship is between the proposed class and, for example, coverage functions. In Section~\ref{sec:strongly}, we will show that any weakly infinite-alternating function can be obtained from an infinite-alternating function by subtracting a nonnegative charge (Corollary \ref{cor:strong2weak}).

%%%%%%%%%%%%%%%%%%%%%%%%%%%%%%%%
\subsection{Examples}
\label{sec:hypergraph}
%%%%%%%%%%%%%%%%%%%%%%%%%%%%%%%%

To give a better understanding of infinite-alternating and weakly infinite-alternating set functions, we present some basic examples.

\begin{ex} \label{ex:constant}
    Let $(J,\cB)$ be a set algebra. Every constant function $\varphi$ on $\cB$ satisfies \eqref{eq:kalt} for all $k\geq 1$, hence it satisfies all requirements of being infinite-alternating apart from $\varphi(\emptyset)=0$. Furthermore, $V_\varphi(A_0;A_1)=0$ for all $A_0, A_1$, hence by induction and \eqref{eq:diminishing}, $V_\varphi(A_0;A_1,\ldots , A_k)=0$ for all $k \geq 1$ and $A_0, A_1, \ldots , A_k$. This guarantees that the assumption $\varphi(\emptyset)=0$ could be relaxed everywhere, as shifting the values of $\varphi$ by a constant does not change the values $V_\varphi(A_0; A_1, \ldots , A_k)$.
\end{ex}

\begin{ex}\label{ex:modular}
    Let $(J,\cB)$ be a set algebra. Every nonnegative charge (or equivalently, every increasing modular set function) $\mu \colon \cB \to \bR$ is infinite-alternating with $V_\mu(A_0; A_1, \ldots , A_k)=0$ for all $k \geq 2$ and pairwise disjoint sets $A_0, A_1, \ldots , A_k$. Indeed, as $\mu$ is modular, $V_\mu(A_0; A_1, A_2)=0$ for all pairwise disjoint sets $A_0, A_1, A_2$, hence by induction and \eqref{eq:diminishing}, $V_\mu(A_0;A_1,\ldots , A_k)=0$ for all $k \geq 2$ and pairwise disjoint sets $A_0, A_1, \ldots , A_k$.
    
    If $\mu$ is a charge (or equivalently, a modular set function), then an analogous computation shows that $\mu$ is weakly infinite-alternating with $V_\mu(A_0; A_1, \ldots , A_k)=0$ for all $k \geq 2$ and pairwise disjoint sets $A_0, A_1, \ldots , A_k$.
\end{ex}

\begin{ex} \label{ex:cutfn}
    Let $G=(V,E)$ be a graph and $w \colon E \to \bR_+$ be a weight function. Then, $e_w$ is infinite-alternating and $d_w$ is weakly infinite-alternating. Furthermore, we have $V_{e_w}(A_0; A_1, \ldots , A_k)=0$ and $V_{d_w}(A_0;A_1,\dots,A_k)=0$ for all $k \geq 3$ and pairwise disjoint sets $A_0, A_1, \ldots , A_k$.  
    
    To see these, note that $e_w=\sum_{\{u,v\} \in E} w(\{u,v\}) \cdot \varphi_{\{u,v\}}$, where $\varphi_A$ is defined as in $\eqref{eq:extremal_def}$. Therefore, $e_w$ is infinite alternating by \cref{prop:extremal},  and $V_{e_w}(A_0; A_1, \ldots , A_k)=0$ holds for all $k \geq 3$ and pairwise disjoint sets $A_0, A_1, \ldots , A_k$ by \cref{lem:k_0}. The statements about $d_w$ follow from the equality $d_w=2e_w-(e_w+i_w)$, since $e_w(X)+i_w(X)=\sum_{x \in X} d_w(x)$ is modular.
\end{ex}

It is not difficult to see that a set function $\varphi$ is infinite-alternating with $V_\varphi(A_0,;A_1,\ldots,A_k)=0$ for $k=1$ if and only if it is identically $0$. \cref{ex:modular} and \cref{ex:cutfn} provide examples of infinite-alternating and weakly infinite-alternating set functions such that $V_\varphi(A_0; A_1, \ldots , A_k)=0$ for all $k \geq 2$ and $k \geq 3$, respectively. Using \cref{lem:k_0}, it is easy to see that the converse is also true, i.e., these examples characterize infinite-alternating and weakly infinite-alternating set functions with $V_\varphi(A_0; A_1, \ldots , A_k)=0$ for all $k \geq 1$, $k \geq 2$ or $k \geq 3$.

\begin{rem}
    It is worth mentioning that hypergraph cut functions are not weakly infinite-alternating. Let $e=\{a,b,c,d\}$ be the only hyperedge of a hypergraph, and let $A_0=\{a\}$, $A_1=\{b\}$, $A_2=\{c\}$, $A_3=\{d\}$. If $d$ denotes the cut function of the hypergraph, then $d(A_0 \cup \bigcup_{i \in K} A_i)=0$ if and only if $K=[3]$, hence
    \begin{equation*}
    V_d(A_0;A_1,A_2,A_3)=1-3+3-0=1.
    \end{equation*}
\end{rem}

\begin{ex} 
    A natural problem is to characterize matroid rank functions that satisfy \eqref{eq:kalt} for some $k$. While every infinite-alternating set function can be obtained as the nonnegative combination of partition matroid rank functions, it turns out that the hierarchy of $k$-alternating matroid rank functions does not give an interesting categorization of matroids. More precisely, let $(M,r)$ be a loopless matroid with rank function $r$. Then, the following are equivalent:
    \begin{enumerate}[label=(\roman*)]\itemsep0em
        \item $r$ is 3-alternating, \label{it:i}
        \item $r$ is infinite-alternating, \label{it:ii}
        \item $M$ is a partition matroid with upper bound $1$ on each partition class. \label{it:iii}
    \end{enumerate}
    
    Here, the implication \ref{it:ii}$\Rightarrow$\ref{it:i} is immediate.

    For \ref{it:i}$\Rightarrow$\ref{it:iii}, suppose first that $M$ has a circuit $C$ with $|C| \geq 3$. Let $A_0=\emptyset$ and $A_1, A_2, A_3$ be a partition of $C$. Then, $r$ is almost the modular set function $r(X)=|X|$ on the set algebra generated by $A_1, A_2, A_3$, except for $r(A_1 \cup A_2 \cup A_3)=r(C)=|C|-1$. This implies $V_r(A_0;A_1,A_2,A_3)=1$, contradicting $r$ being $3$-alternating. Therefore, each circuit of $M$ has size at most 2, which is equivalent to $M$ being a partition matroid with upper bound 1 on each partition class.

    Finally, for \ref{it:iii}$\Rightarrow$\ref{it:ii}, let $P_1, \ldots , P_q$ denote the partition classes defining $M$. Then, $r=\sum_{i=1}^q \varphi_{P_i}$, where $\varphi_{A}$ is defined as in \eqref{eq:extremal_def}.  Hence $r$ is infinite-alternating by \cref{prop:extremal}.     
\end{ex}

%%%%%%%%%%%%%%%%%%%%%%%%%%%%%%%%
\subsection{Difference of infinite-alternating set functions}
\label{sec:infaltdiff}
%%%%%%%%%%%%%%%%%%%%%%%%%%%%%%%%

We have already seen in the introduction that any submodular set function over a finite set has a monotonic diff-decomposition. Here we strengthen this result by showing that any set function $\varphi\colon 2^J \to \bR$ over a finite domain can be expressed as the difference of two infinite-alternating set functions. We show two different proofs for this result. The first proof uses linear algebraic techniques, and it provides a canonical decomposition in a certain sense. In contrast, the second proof is more elementary and yields a decomposition in which $\varphi_2$ is always the same for any $\varphi$ up to multiplication by a constant.

\begin{thm} \label{thm:finite_diff-decomp}
    Let $J$ be a finite set and $ \varphi\colon 2^J \to \bR$ be an arbitrary set function with $\varphi(\emptyset)=0$. Then, there exist infinite-alternating set functions $\varphi_1, \varphi_2\colon 2^J \to \bR$ such that $\varphi=\varphi_1-\varphi_2$.
\end{thm}
\begin{proof}[First proof]
 A set function $\psi \colon 2^J \to \bR$ with $\psi(\emptyset)=0$ can be thought of as a vector of size $\bR^{2^{|J|}-1}$ with entries corresponding to the values of $\psi$ on nonempty subsets. Let $\mathbf{P} \in \bR^{(2^{|J|}-1) \times (2^{|J|}-1)}$ be the matrix with rows and columns indexed by the nonempty subsets of $J$, where $\mathbf{P}_{X,Y}=\mathbbm{1}(X \cap Y \neq \emptyset)$; here $\mathbbm{1}$ denotes an indicator function. Let us also define $\mathbf{N} \in \bR^{(2^{|J|}-1) \times (2^{|J|}-1)}$ as the matrix with $\mathbf{N}_{X,Y}=(-1)^{|X \cap Y|-1} \cdot \mathbbm{1}(X \cup Y =J)$. An easy calculation shows that $\mathbf{N}=\mathbf{P}^{-1}$, see e.g. \cite{lovasz2023submodular}. The set function $\mathbf{P}\psi$ is defined by the appropriate matrix multiplication, with $\mathbf{P}\psi(\emptyset)=\mathbf{N}\psi(\emptyset)=0$. Proposition \ref{prop:extremal} implies that $\varphi$ is infinite-alternating if and only if it can be expressed as $\varphi=\mathbf{P}\alpha$ with a set function $\alpha:~2^J\setminus \{\emptyset\}\to\R_+$; this, in turn, is equivalent with $\mathbf{N}\varphi\ge 0$.  In particular, the nonsingularity of $\mathbf{P}$ implies \cref{rem:extremal_unique}.
 
   Let $\varphi_1=\mathbf{P}\left( (\mathbf{N}\varphi)^+ \right)$ and $\varphi_2=\mathbf{P}\left( (\mathbf{N}\varphi)^- \right)$. Then, 
    \begin{align*}
        \varphi_1-\varphi_2={}&{} \mathbf{P}\left( (\mathbf{N}\varphi)^+ \right)-\mathbf{P}\left( (\mathbf{N}\varphi)^- \right) \\
        = {}&{} \mathbf{P}(\mathbf{N}\varphi) \\
        ={}&{} \varphi,
    \end{align*}
    where $\varphi_1$ and $\varphi_2$ are infinite-alternating by \cref{prop:extremal}, as the multiplication of $\mathbf{P}$ by a nonnegative vector is a linear combination of functions $\varphi_A$ given by \eqref{eq:extremal_def} with nonnegative coefficients.
\end{proof}

\begin{proof}[Second proof.]
    Let us define
    \begin{equation*}
        m\coloneqq\max\Bigl\{
        V_\varphi(A_0;A_1, \ldots , A_k) \
        %\sum_{K \subseteq [k]} (-1)^{|K|}\varphi\bigl(A_0 \cup \bigcup_{i \in K} A_i\bigr)
        \bigl\vert \ k\in[|J|], A_0,\dots,A_k\subseteq J\ \text{are pairwise disjoint},A_i\neq\emptyset\ \text{for $i\in[k]$}\Bigr\}.
    \end{equation*}    
    Let $\varphi_2\coloneqq m \cdot \sum_{\emptyset\neq A \subseteq J} \varphi_A$, 
    and let $\varphi_1\coloneqq \varphi+\varphi_2$. By \cref{prop:extremal}, the function $\varphi_2$ is infinite-alternating. By \cref{lem:extremal 2} and the definition of $m$,
    \begin{align*}
        V_{\varphi_1}(A_0;A_1, \ldots , A_k)= {}&{} V_\varphi(A_0;A_1, \ldots , A_k)+V_{\varphi_2}(A_0;A_1, \ldots , A_k) \\
        \leq {}&{} 0
    \end{align*}
    for all $1 \leq k \leq |J|$ and pairwise disjoint sets $A_0, A_1, \ldots , A_k$. Consequently, $\varphi_1$ is infinite-alternating by \cref{lem:weak_strong_alternating_connection}.
\end{proof}

%%%%%%%%%%%%%%%%%%%%%%%%%%%%%%%%
\section{General submodular set functions}
\label{sec:main}
%%%%%%%%%%%%%%%%%%%%%%%%%%%%%%%%

The aim of this section is to answer both questions about the existence of monotonic decompositions in the negative. In Section~\ref{sec:finite}, we reduce the existence of such decompositions to the finite setting, with the additional requirement that the decompositions are $c$-bounded for some universal constant $c$. Using this connection, we show that there exists a bounded submodular set function with no monotonic sum-decomposition in Section~\ref{sec:sum}, and with no monotonic diff-decomposition in Section~\ref{sec:difference}. 

%%%%%%%%%%%%%%%%%%%%%%%%%%%%%%%%
\subsection{From infinite to finite domain}
\label{sec:finite}
%%%%%%%%%%%%%%%%%%%%%%%%%%%%%%%%

The idea behind reducing decomposition problems to finite domain is a standard compactness argument. For a set algebra $(J, \cB)$ we can consider the space of set functions $\{f\colon \cB \to \bR\}$ as the product topological space $\prod_{A \in \cB} \bR$. Furthermore, for any $r\colon \cB \to \bR_+$, the space of functions 
\begin{equation*}
   \Big\{f\colon \cB \to \bR \ \big| \  |f(A)| \leq r(A) \text{ for all } A \in \cB \Big\}=\prod_{A \in \cB} [-r(A),r(A)] 
\end{equation*} is compact, as the product of compact sets is compact. We will use that a topological space $X$ is compact if and only if for any collection $\cC$ of closed subsets of $X$ with $\bigcap_{C \in \cC'} C \neq \emptyset$ for all finite subcollection $\cC' \subseteq \cC$, $\bigcap_{C \in \cC} C \neq \emptyset$ also holds.

Let $(J_1, \cB_1)$ and $(J_2, \cB_2)$ be two set algebras. The functions $\varphi_1\colon \cB_1 \to \bR$ and $\varphi_2\colon \cB_2 \to \bR$ are \emph{isomorphic} if there exists a bijection $f\colon J_1 \rightarrow J_2$ such that $f$ and $f^{-1}$ are measurable and $\varphi_2(f(X))=\varphi_1(X)$ for all $X\in \cB_1$.
By a \emph{property} of set functions we mean an isomorphism invariant class of set functions.
We call a property $P$ of set functions \emph{finitary} if a function $\varphi\colon \cB \to \bR$ is in $P$ if and only if $\varphi/\cQ$  is in $P$ for any finite measurable partition $\cQ$ of $J$. Furthermore, call a finitary property $P$ \emph{closed}, if for any set algebra $(J, \cB)$ and any finite measurable partition $\cQ$ of $J$, the set of functions
$\{f\colon \cB \to \bR \ | \ f/\cQ \in P\}$
is closed. Note that this is more general than requiring $\{f\colon\cB \to \bR \ | \ f \in P\}$ to be closed, however for reasonable properties, they are equivalent. 

Let $(J,\cB)$ be a set algebra and $P_1,P_2$ be properties of set functions. We say that a set function $\psi\colon \cB\to\bR$ is \emph{$(P_1,P_2)$-decomposable} if there exist functions $\psi_{1}, \psi_2\colon \cB \to \bR$ of property $P_1$ and $P_2$, respectively, such that $\psi=\psi_1+\psi_2$.

\begin{thm}\label{thm:p1p2} 
Let $P_1$ and $P_2$ be finitary, closed properties of set functions and $(J,\cB)$ be a set algebra. A function $\psi\colon \cB \to \bR$  is $(P_1,P_2)$-decomposable if and only if there exists a function $r\colon \cB\to\bR$ such that for all finite measurable partition $\cQ$, $\psi/\cQ$ has a $(P_1,P_2)$-decomposition $\psi/\cQ=\psi^\cQ_1+\psi^\cQ_2$ with $|\psi^\cQ_1(A)|\leq r(A)$ for all $A\in\cB_\cQ$.
\end{thm}
\begin{proof} 
For the `if' direction, we define
\begin{equation*}
    S\coloneqq \{\psi_{1}\colon \cB \to \bR \mid \psi_1 \in P_1, \psi-\psi_1\in P_2,\, |\psi_1(A)| \le r(A)\ \text{for all $A \in \cB$}\}.
\end{equation*} 
Similarly, for a finite measurable partition $\cQ$ let
\begin{equation*}
    S/\cQ \coloneqq \{\psi_{1}\colon \cB \to \bR \mid \psi_1/\cQ \in P_1, (\psi-\psi_1)/\cQ \in P_2,\, |(\psi_1/\cQ)(A)| \le r(A)\ \text{for all $A \in \cB_\cQ$}\}.
\end{equation*} 
We need to show that $S$ is nonempty. Observe that, since $P_1$ and $P_2$ are closed, $S/\cQ$ is closed for all finite measurable partitions $\cQ$, and since $P_1$ and $P_2$ are finitary, $S=\bigcap_\cQ S/\cQ$ where the intersection is taken over all finite measurable partitions. For finitely many partitions $\cQ_1,\ldots , \cQ_n$, we have $\cap_{i=1}^{n} S/\cQ_i \supseteq S/(\bigwedge_{i=1}^n \cQ_i)$ where $\bigwedge_{i=1}^n \cQ_i$ is the least common refinement of the partitions. By our assumption, $S/(\bigwedge_{i=1}^n \cQ_i)$ is nonempty. Therefore, the compactness of $\{f\colon \cB \to \bR \mid |f(A)| \le r(A) \text{ for all $A \in \cB$}\}$ implies that $S=\bigcap_{\cQ} S/\cQ$ is also nonempty.

For the `only if' direction, let $\psi=\psi_1+\psi_2$ be a $(P_1,P_2)$-decomposition. Then, for any $A\in\cB$,  we can choose $r(A)$ to be $|\psi_1(A)|$.
\end{proof}

\begin{cor} \label{cor:constant}
Let $(J,\cB)$ be a set algebra and $\psi\colon\cB\to\bR$ a bounded set function with $\psi(\emptyset)=0$. If there exists a $c\geq 1$ such that for every finite measurable partition $\cQ$, the quotient $\psi/\cQ$ can be written as the difference of an infinite-alternating function $\varphi$ and a nonnegative charge $\mu$ with $\max\{\|\varphi\|,\|\mu\|\}\leq c\cdot\|\psi/\cQ\|$, then $\psi$ can be written as the difference of an infinite-alternating function and a nonnegative charge. 
\end{cor}
\begin{proof}
Assume that there exists $c\geq 1$ satisfying the conditions of the corollary. For each $A\in\cB$, set $r(A)$ to be $c\cdot\|\psi\|$. Since `infinite alternating' and `nonpositive, modular' are finitary and closed properties of set functions, the existence of the decomposition in question follows by the `if' direction of Theorem~\ref{thm:p1p2}. 
\end{proof}

%%%%%%%%%%%%%%%%%%%%%%%%%%%%%%%%
\subsection{Monotonic sum-decompositions}
\label{sec:sum}
%%%%%%%%%%%%%%%%%%%%%%%%%%%%%%%%

For every $c\in\bR_+$, we construct a submodular set function over a finite domain having no $c$-bounded monotonic sum-decomposition. We will need the following two technical lemmas.

\begin{lem}\label{lem:modular1}
Let $(J,\cB)$ be a set algebra and $\varphi \colon \cB \to \bR$ with $\varphi(\emptyset)=0$. Assume that $\varphi$ is modular on a pair $A,B\in\cB$ and let $\varphi=\varphi_1+\varphi_2$ be a monotonic sum-decomposition. Then, $\varphi_1$ and $\varphi_2$ are also modular on the pair $A,B$.
\end{lem}
\begin{proof}
    Using the fact that $\varphi$ is modular on the pair $A,B$ and $\varphi=\varphi_1+\varphi_2$, we get
    \begin{align*}
    \varphi(A)+\varphi(B)
    {}&{}= \varphi_1(A)+\varphi_2(A)+\varphi_1(B)+\varphi_2(B)\\
    {}&{}\geq \varphi_1(A\cap B)+\varphi_1(A \cup B)+\varphi_2(A\cap B)+\varphi_2(A \cup B)\\
    {}&{}=\varphi(A\cap B)+\varphi(A\cup B)\\
    {}&{}=\varphi(A)+\varphi(B).
    \end{align*}
    Therefore, equality holds throughout, implying $\varphi_i(A)+\varphi_i(B)=\varphi_i(A\cap B)+\varphi_i(A\cup B)$ for $i\in[2]$. 
\end{proof}

\begin{lem}\label{lem:modular2}
    Let $J$ be a finite set and $\varphi \colon 2^J \to \bR$ submodular. Let $X\subseteq Y\subseteq J$ such that $\varphi$ is modular on any pair $A,B\in \cB$ with $X \subseteq A, B \subseteq Y$. Then, $\varphi(T)=\varphi(X) +\sum_{t \in T \setminus X} (\varphi(X+t)-\varphi(X))$ for any $X \subseteq T \subseteq Y$.
\end{lem}
\begin{proof}
    We prove the statement by induction on $|T|$. The statement clearly holds when $|T|=|X|$ or $|T|=|X|+1$. Assume that $|T| \geq |X|+2$ and that we already proved the statement for all $X \subseteq T' \subseteq Y$ with $|T'|<|T|$. Pick an arbitrary element $x \in T \setminus X$. Using that $\varphi$ is modular on the pair $X+x,T-x$ and the induction hypothesis, we get
    \begin{align*}
        \varphi(T)
        {}&{}=\varphi(T-x)+\varphi(X+x)-\varphi(X)\\
        {}&{}=\varphi(X) +\sum_{t \in T \setminus X} \bigl(\varphi(X+t)-\varphi(X)\bigr),
    \end{align*}
    concluding the proof of the lemma.
\end{proof}

We now prove that there is no universal constant $c$ for which every submodular set function over a finite domain has a $c$-bounded monotonic sum-decomposition.

\begin{thm}\label{thm:finite_decomposition}
For any $c\in\bR_+$, there exists a submodular set function over a finite domain with value $0$ on the empty set that does not admit a $c$-bounded monotonic sum-decomposition.
\end{thm}
\begin{proof}
Let $n$ be a positive integer, let $J=\{a_1, a_2, \ldots ,a_{2n}\}$ and set $A=\{a_1,a_2, \ldots , a_n\}$. We define a set function $\varphi\colon 2^J \to \bR$ by
\begin{equation}
\varphi(X) = \begin{cases}
0 &\text{if $X \subseteq A$ or $A \subseteq X$,}\\
1 &\text{otherwise.}
\end{cases}\label{eq:sum}    
\end{equation}

\begin{cl}\label{cl:sub}
    $\varphi$ is submodular.
\end{cl}
\begin{claimproof}
    Let $X,Y \subseteq J$ be arbitrary sets. We need to show that $\varphi(X)+\varphi(Y) \geq \varphi(X \cap Y)+\varphi(X \cup Y)$. The image of $\varphi$ is $\{0,1\}$, hence we distinguish three cases.
    
    The inequality clearly holds if $\varphi(X)=\varphi(Y)=1$.
    
    If $\varphi(X)=0$ and $\varphi(Y)=1$, then either $X \subseteq A$ implying $\varphi(X \cap Y)=0$, or $A \subseteq X$ implying $\varphi(X \cup Y)=0$. The case when $\varphi(X)=1$ and $\varphi(Y)=0$ can be analyzed in a similar way.

    If $\varphi(X)=\varphi(Y)=0$, then either $A \subseteq X,Y$,  or $X,Y \subseteq A$,  or $X \subseteq A \subseteq Y$, or $Y \subseteq A \subseteq X$. In all four cases, we have $\varphi(X\cap Y)=\varphi(X\cup Y)=0$, finishing the proof of the claim.
\end{claimproof}

Let $\varphi=\varphi_1+\varphi_2$ be an arbitrary monotonic sum-decomposition. 

\begin{cl}\label{cl:n}
$\varphi_1(J)\geq n$.    
\end{cl}
\begin{claimproof}
Using Lemma~\ref{lem:modular1} for the sets $A-a_i$ and $\{a_{n+i}\}$, the monotonicity of $\varphi_1$, and $\varphi_1 \geq \varphi$ we obtain
\begin{align*}
\varphi_1(A+a_{n+i}) 
{}&{}\geq \varphi_1(A+a_{n+i}-a_i)\\
{}&{}=\varphi_1(A-a_i)+\varphi_1(a_{n+i})-\varphi_1(\emptyset) \\
{}&{}\geq \varphi_1(A-a_i)+1    
\end{align*}
for all $i\in[n]$. Note that the pair $(\emptyset, A)$ satisfies the conditions of Lemma~\ref{lem:modular2} for $\varphi$. Therefore, by Lemma~\ref{lem:modular1},  the same holds for $\varphi_1$ as well. This implies $\varphi_1(A-a_i)=\varphi_1(A)-\varphi_1(a_i)$ and $\varphi_1(A)=\sum_{i=1}^n \varphi_1(a_i)$. On the other hand, the pair $(A,J)$ also satisfies the conditions of Lemma~\ref{lem:modular2}, thus we get
\begin{align*}
    \varphi_1(J)
    {}&{}=\varphi_1(A)+\sum_{i=1}^n \bigl(\varphi_1(A+a_{n+i})-\varphi_1(A)\bigr)\\
    {}&{}\geq \varphi_1(A)+\sum_{i=1}^n \bigl(\varphi_1(A-a_i)+1-\varphi_1(A)\bigr)\\
    {}&{}= \varphi_1(A)+\sum_{i=1}^n \bigl(\varphi_1(A)-\varphi_1(a_i)+1-\varphi_1(A)\bigr)\\
    {}&{}=n,
\end{align*}
completing the proof of the claim.    
\end{claimproof}

Since $n$ was chosen arbitrarily, the theorem follows by Claim~\ref{cl:n}. 
\end{proof}

\begin{cor}\label{cor:sum}
There exists a bounded submodular set function with value $0$ on the empty set that does not admit a monotonic sum-decomposition.
\end{cor}
\begin{proof}
Let us define a set function over the subsets of $\bZ$ as
\begin{equation*}
    \psi(X)=\begin{cases}
        0 & \text{if $X\subseteq 2\cdot\bZ$ or $2\cdot\bZ\subseteq X$,}\\
        1 & \text{otherwise.}
    \end{cases}
\end{equation*}
Then, $\psi$ is a bounded submodular set function -- this can be seen similarly as in the proof of Claim~\ref{cl:sub}. It is not difficult to check that for every $n\in\bZ_+$, there exists a finite partition $\cQ$ of $\bZ$ with $n$ partition classes such that the set function $\varphi$ defined in \eqref{eq:sum} is equal to $\psi/\cQ$. By Theorems~\ref{thm:p1p2} and~\ref{thm:finite_decomposition}, $\psi$ does not admit a monotonic sum-decomposition, since the value of $r(J)$ cannot be set to any finite value.
\end{proof}

\begin{rem}
The question naturally arises whether symmetric submodular set functions have monotonic sum-decompositions. Let $\varphi$ be the set function appearing in the proof of Theorem~\ref{thm:finite_decomposition} and let $\tilde{\varphi}$ denote its symmetrized version, that is, $\tilde{\varphi}(X)\coloneqq\varphi(X)+\varphi(J\setminus X)$ for all $X\in2^J$. Then, $\tilde{\varphi}$ is a symmetric submodular set function, and an analogous proof shows that for any monotonic sum-decomposition $\tilde{\varphi}=\varphi_1+\varphi_2$, we have $\varphi_1(J)\geq n$.
\end{rem}

%%%%%%%%%%%%%%%%%%%%%%%%%%%%%%%%
\subsection{Monotonic diff-decompositions}
\label{sec:difference}
%%%%%%%%%%%%%%%%%%%%%%%%%%%%%%%%

We now show that the infinite counterpart of Theorem~\ref{thm:finite_diff-decomp} does not hold. In particular, we show that a submodular set function does not necessarily have a monotonic diff-decomposition.

\begin{thm}\label{thm:finite_decomposition_diff}
For any $c\in\bR_+$, there exists a submodular set function over a finite domain with value $0$ on the empty set that does not admit a $c$-bounded monotonic diff-decomposition.
\end{thm}
\begin{proof}
Let $n$ be a positive integer and let $J=\{a_1, a_2, \ldots ,a_n\}$. We define a set function $\varphi\colon 2^J \to \bR$ by
\begin{equation}
\varphi(X) = \begin{cases}
-1 &\text{if $X = J$,}\\
0 &\text{otherwise.}
\end{cases}\label{eq:diff}    
\end{equation}
The function $\varphi$ is clearly submodular. Let $\varphi=\varphi_1-\varphi_2$ be an arbitrary monotonic diff-decomposition. 

\begin{cl}\label{cl:nn}
$\varphi_2(J)\geq n$.    
\end{cl}
\begin{claimproof}
For any $a \in J$ and any $A \subseteq J$ with $a \notin A$, we have 
\begin{align*}
\varphi_2(A+a)-\varphi_2(A)
{}&{}\geq \varphi_2(J)-\varphi_2(J-a)\\
{}&{}=\bigl(\varphi_1(J)-\varphi_1(J-a)\bigr)-\bigl(\varphi(J)-\varphi(J-a)\bigr)\\
{}&{}\geq \varphi(J-a)-\varphi(J)\\
{}&{}=1, 
\end{align*}
where the first inequality holds by the submodularity of $\varphi_2$, the first equality holds by $\varphi=\varphi_1-\varphi_2$, the second inequality holds by the monotonicity of $\varphi_1$, and the last equality holds by the definition of $\varphi$. Thus
\begin{align*}
\varphi_2(J)
{}&{}=\varphi_2(a_1)-\varphi_2(\emptyset)+\sum_{i=2}^n \varphi_2\Bigl(\bigl(\{a_1, a_2, \ldots , a_i\}\bigr)-\varphi_2\bigl(\{a_1, a_2, \ldots , a_{i-1}\}\bigr)\Bigr)\\
{}&{}\geq n,    
\end{align*}
finishing the proof of the claim.    
\end{claimproof}

Since $n$ was chosen arbitrarily, the theorem follows by Claim~\ref{cl:nn}. 
\end{proof}

\begin{cor}\label{cor:diff}
There exists a bounded submodular set function with value $0$ on the empty set that does not admit a monotonic diff-decomposition.
\end{cor}
\begin{proof}
Let us define a set function over the subsets of $\bZ$ as
\begin{equation*}
    \psi(X)=\begin{cases}
        -1 & \text{if $X=\bZ$,}\\
        0 & \text{otherwise.}
    \end{cases}
\end{equation*}
Then, $\psi$ is a bounded submodular set function, and it is not difficult to check that for every $n\in\bZ_+$, there exists a finite partition $\cQ$ of $\bZ$ with $n$ partition classes such that the set function $\varphi$ defined in \eqref{eq:diff} is equal to $\psi/\cQ$. By Theorems~\ref{thm:p1p2} and~\ref{thm:finite_decomposition_diff}, $\psi$ does not admit a monotonic diff-decomposition, since the value of $r(J)$ cannot be set to any finite value.
\end{proof}

\begin{rem}
Similarly to the case of sum-decompositions, the question naturally arises whether symmetric submodular set functions have monotonic diff-decompositions. Let $\varphi$ be the set function appearing in the proof of Theorem~\ref{thm:finite_decomposition_diff} and let $\tilde{\varphi}$ denote its symmetrized version shifted to satisfy $\tilde{\varphi}(\emptyset)=0$, that is, $\tilde{\varphi}(X)\coloneqq\varphi(X)+\varphi(J\setminus X)+1$ for all $X\in2^J$. Then, $\tilde{\varphi}$ is a symmetric submodular set function, and an analogous proof shows that for any monotonic diff-decomposition $\tilde{\varphi}=\varphi_1-\varphi_2$ we have $\varphi_1(J)\geq n$.
\end{rem}

%%%%%%%%%%%%%%%%%%%%%%%%%%%%%%%%
\section{Weakly infinite-alternating set functions}
\label{sec:strongly}
%%%%%%%%%%%%%%%%%%%%%%%%%%%%%%%%

Note that monotone decompositions are not interesting for infinite-alternating set functions since those are also increasing. However, weakly infinite-alternating set functions differ from infinite-alternating ones precisely in relaxing the monotonicity condition. Our main result is proving that, unlike general submodular set functions, weakly infinite-alternating set functions admit monotonic sum- and diff-decompositions, where $\varphi_1$ can be chosen to be infinite-alternating while $\varphi_2$ can be chosen to be a charge. In particular, we show that weakly infinite-alternating set functions over a finite domain have $7$-bounded monotonic sum- and diff-decompositions, which in turn implies the desired result. We first show that every weakly infinite-alternating set function over a finite set $J$ can be obtained as the difference of an infinite-alternating set function and a nonnegative charge, a result that is interesting on its own.

\begin{lem}\label{lem:quasi_decomposition}
    Let $J$ be a finite set. A function $\psi \colon 2^J \to \bR$ is weakly infinite-alternating if and only if it can be written as $\psi=\varphi-\mu$,  where $\varphi$ is infinite-alternating and $\mu$ is a nonnegative charge.
\end{lem}
\begin{proof}
    By \cref{ex:modular}, $\mu$ satisfies \eqref{eq:kalt} with equality for $k\geq 2$ on pairwise disjoint sets, implying that $\varphi-\mu$ is weakly infinite-alternating.

    For the other direction, assume that $\psi$ is weakly infinite-alternating. Define $\mu(X)=2 |X| \cdot \|\psi\|$ for all $X \subseteq J$, which is clearly a nonnegative charge. By definition, $\varphi=\psi+\mu$ is increasing. Furthermore, both $\psi$ and $\mu$ are weakly infinite-alternating. These together imply that $\varphi$ is in fact an infinite-alternating set function.
\end{proof}

Our main result is the following.

\begin{thm}\label{thm:weakly}
    Every weakly infinite-alternating set function over a finite domain has $7$-bounded monotonic sum- and diff-decompositions. 
\end{thm}
\begin{proof}
Let $J$ be a finite set and $\psi\colon2^J\to\bR$ be a weakly infinite-alternating set function. Let $\psi=\varphi-\mu$ be a decomposition of $\psi$ provided by Lemma~\ref{lem:quasi_decomposition}. By Proposition~\ref{prop:extremal}, $\varphi$ can be written in the form $\varphi=\sum_{\emptyset \neq A\subseteq J}\alpha_A\varphi_A$, where $\varphi_A$ is defined as in \eqref{eq:extremal_def}. Since $\varphi_{\{a\}}$ is modular for any $a\in J$, by decreasing $\alpha_{\{a\}}$ and increasing $\mu$ for all set containing $a$ by the same constant, we may assume that $\mu(a) \cdot \alpha_{\{a\}}=0$ for all $a \in J$. Note that with this additional constraint, the decomposition $\psi=\varphi-\mu$ is unique by \cref{rem:extremal_unique}.  We give two lower bounds for $\|\psi\|$.

\begin{cl}\label{cl:bound1}
$\|\psi\|\geq |\varphi(J)-\mu(J)|$.
\end{cl}
\begin{claimproof}
    The statement immediately follows from the equality $\psi(J)=\varphi(J)-\mu(J)$.
\end{claimproof}

\begin{cl}\label{cl:bound2}
$\|\psi\|\geq \frac{3}{4}\cdot \varphi(J)-\frac{1}{2}\cdot \mu(J)$.
\end{cl}
\begin{claimproof}
    We prove the statement using a probabilistic method. Let $J'$ be the support of $\mu$, i.e., $J'=\{a \in J \mid \mu(a) \neq 0\}$, and let $J''=J \setminus J'$. Furthermore, let $\cA'$ be the family of nonempty sets contained in $J'$, i.e., $\cA'=\{A \subseteq J' \mid A \neq \emptyset\}$, and let $\cA''=2^J\setminus \{\cA'\cup\{\emptyset\}\}$. Note that, by our assumption and by the definition of $J'$, $\alpha_{\{a\}}=0$ for all $\{a\} \in \cA'$.
    
    Define $\varphi'=\sum_{A \in \cA'} \alpha_A \varphi_{A}$. Select a random subset $X$ of $J'$ by including each element with probability $\frac{1}{2}$, independently of other elements. The set $X$ intersects every $A \subseteq J'$ of size at least two with probability at least $\frac{3}{4}$, hence
    \begin{align*}
    \bE\bigl(\varphi'(X)-\mu(X)\bigr) 
    {}&{}\geq \tfrac{3}{4}\cdot  \varphi'(J')-\tfrac{1}{2}\cdot \mu(J')\\
    {}&{}=\tfrac{3}{4}\cdot  \varphi'(J)-\tfrac{1}{2}\cdot \mu(J).    
    \end{align*}
    It follows that there exists a set $T \subseteq J'$ satisfying $\varphi'(T)-\mu(T) \geq \frac{3}{4}\cdot  \varphi'(J)-\frac{1}{2}\cdot \mu(J)$. Since all sets $A \in \cA''$ intersects $J''$, we have $\varphi(T\cup J'')=\varphi'(T)+\varphi(J)-\varphi'(J)$. This, together with $\mu(J'')=0$ and $\varphi'\leq\varphi$, implies 
    \begin{align*}
    \|\psi\|
    {}&{}= \max_{A \subseteq J} |\psi(A)|\\ 
    {}&{} \geq \psi(T \cup J'')\\
    {}&{} =\varphi(T \cup J'')-\mu(T \cup J'')\\
    {}&{} =\bigl(\varphi'(T)-\mu(T)\bigr)+\bigl(\varphi(J)-\varphi'(J)-\mu(J'')\bigr)\\
    {}&{} \geq \tfrac{3}{4}\cdot \varphi'(J)-\tfrac{1}{2}\cdot \mu(J)+\varphi(J)-\varphi'(J)\\
    {}&{} \geq \tfrac{3}{4}\cdot \varphi(J)-\tfrac{1}{2}\cdot \mu(J),
    \end{align*} 
concluding the proof of the claim.
\end{claimproof}

To finish the proof of the theorem, we distinguish two cases. If $\mu(J) \geq \frac{7}{6}\cdot \varphi(J)$, then Claim~\ref{cl:bound1} implies
\begin{equation*}
\varphi(J)\leq 6\cdot|\varphi(J)-\mu(J)|\leq 6\cdot \|\psi\|. 
\end{equation*}
If $\mu(J) \leq \frac{7}{6}\cdot  \varphi(J)$, then Claim~\ref{cl:bound2} implies
\begin{equation*}
\varphi(J)\leq 6\cdot\bigl(\tfrac{3}{4}\cdot\varphi(J)-\tfrac{1}{2}\cdot\mu(J)\bigr)\leq 6\cdot \|\psi\|.
\end{equation*}
Note that $\|\mu\|=\|\psi-\varphi\|\leq \|\psi\|+\|\varphi\|\leq 7\cdot \|\psi\|$. Since $\varphi$ is infinite-alternating and $\mu$ is nonnegative charge, the decomposition $\psi=\varphi-\mu$ is a $7$-bounded monotonic sum-decomposition as well as a $7$-bounded monotonic diff-decomposition. This finishes the proof of the theorem.
\end{proof}

By Corollary~\ref{cor:constant} and Theorem~\ref{thm:weakly}, we get the following result for the infinite case.

\begin{cor}\label{cor:strong2weak}
    Every bounded weakly infinite-alternating set function can be written as the difference of an infinite-alternating set function and a nonnegative charge.
\end{cor}

\begin{rem}
It is worth checking that the set functions defined in \eqref{eq:sum} and \eqref{eq:diff} are not weakly infinite-alternating. Indeed, if $\varphi$ is the function defined in \eqref{eq:sum} for $n \ge 3$, then \[V_\varphi(\emptyset; \{a_1, a_{n+1}\}, \{a_2, a_{n+2}\}, \{a_3, a_4, \dots, a_n\}) = 1,\] showing that $\varphi$ is not weakly 3-alternating. Similarly, if $\varphi$ is the function defined in \eqref{eq:diff} for $n \ge 3$, then for a partition $\{a_1,\dots, a_n\} = A_1 \cup A_2 \cup A_3$ we have $V_\varphi(\emptyset; A_1, A_2, A_3) = 1$, showing that $\varphi$ is not weakly 3-alternating.

\end{rem}

%%%%%%%%%%%%%%%%%%%%%%%%%%%%%%%%
\section{Weighted cut functions}
\label{sec:cut}
%%%%%%%%%%%%%%%%%%%%%%%%%%%%%%%%    

The probably most fundamental examples of non-monotone submodular set functions are the weighted cut functions of undirected graphs. Let $G=(V,E)$ be an undirected graph and $w\colon E \to \bR_+$ be a weight function. Note that $\|d_w\|=\max_{X\subseteq V}d_w(X)$ is the maximum weight of a cut in $G$, a parameter that is of particular combinatorial interest, and whose determination is the so-called \textsc{Max-Cut} problem. \textsc{Max-Cut} is APX-hard in general, roughly meaning that it is NP-hard to approximate the optimum value to within an arbitrary constant factor. On the positive side, the problem admits a simple greedy $2$-approximation: take an arbitrary partition of the vertices into two parts, and repeatedly move a vertex from one side of the partition to the other if this step increases the total weight of the cut. Upon termination, for each vertex the total weight of edges connecting it to the other partition class is at least as large as the total weight of edges connecting it to its own class. This implies that the weight of the cut is at least $w(E)/2$. The best known approximation ratio of $1.139$ is due to Goemans and Williamson~\cite{goemans1995improved} using semidefinite programming and randomized rounding, which was shown to be best possible assuming that the Unique Games Conjecture is true by Khot, Kindler, Mossel, and O'Donnell~\cite{khot2007optimal}.

Example~\ref{ex:cutfn} shows that $d_w$ is weakly infinite-alternating, hence Theorem~\ref{thm:weakly} implies the existence of $7$-bounded monotonic sum- and diff-decompositions of $d_w$. On the other hand, for a $c$-bounded monotonic sum-decomposition $d_w=\varphi_1+\varphi_2$, the value of $\varphi_1(V)=|\varphi_2(V)|$ provides a $c$-approximation of the maximum cut. Thus, in the light of the results of Khot et al., we cannot expect to construct a better than $1.139$-bounded monotonic sum-decomposition (in the sense of a polynomial-time algorithm returning the value of $\varphi_1(S)$ for any $S\subseteq V$). Similar assertion can be made about diff-decomposition.

It is easy to see that for cut functions, the constant $c$ can be reduced from $7$ to $4$ diff-decompositions, and even to $2$ for sum-decompositions. Recall that $d_w=e_w-i_w$, where $e_w$ is an increasing submodular while $i_w$ is an increasing supermodular set function. Furthermore, an easy computation shows that $e_w(X)+i_w(X)=\sum_{v\in X}d_w(v)$ for all $X\subseteq V$, implying that $e_w+i_w$ is a nonnegative charge. Therefore, $d_w=2 e_w-(e_w+i_w)$ provides a monotonic  diff-decomposition of the weighted cut function. By the above greedy algorithm, $\|d_w\|\geq w(E)/2$ while $\|e_w\|\leq w(E)$ and $\|i_w\|\leq w(E)$ hold by definition, thus the decomposition is $4$-bounded. Using a similar reasoning, $d_w=e_w+(-i_w)$ where both $e_w$ and $-i_w$ are submodular, $e_w$ is increasing and $-i_w$ is decreasing, leading to a $2$-bounded monotonic sum-decomposition.

When bounding the value of $c$, the ratio of the maximum weight of a cut to the total weight of the edges plays an important role: if the total weight of the edges is at most $c$ times the maximum weight of a cut, then the constructions above give rise to $c$-bounded monotonic sum- and $2c$-bounded monotonic diff-decompositions of $d_w$. In the unweighted setting, this ratio is well-studied in the extremal graph theory literature and is often referred to as the \emph{bipartite density} $b(G)$ of $G$. That is, $b(G)=\max\{d(X)/|E|\mid X\subseteq V\}$. It was observed by Erdős~\cite{erdos1967even} that $b(G)\ge 1/2$ for every simple graph $G$. He also showed that the lower bound $1/2$ cannot be replaced by any larger real number even if we consider very restricted families such as graphs of large girth.

\begin{prop}[Erdős]\label{prop:erdos}
For any $\varepsilon>0$, there exists a triangle-free simple graph $G$ with $b(G)\leq 1/2+\varepsilon$.    
\end{prop}

For other special classes of graphs, better bounds are known. Staton~\cite{staton1980edge} and Locke~\cite{locke1982maximum} showed that if $G$ is cubic and different from $K_4$, then $b(G)\geq 7/9$. For triangle-free cubic graphs, Bondy and Locke~\cite{zhu2009bipartite} improved the bound to $b(G)\geq 4/5$. Zhu~\cite{zhu2009bipartite} proved that, apart from a few exceptions, $b(G)\geq 17/21$ for 2-connected triangle-free subcubic graphs. 

%%%%%%%%%%%%%%
\subsection{Triangle-free graphs}
\label{sec:triangle_free}
%%%%%%%%%%%%%%

We have already seen that $d_w=e_w-i_w$ gives a $2$-bounded monotonic sum-decomposition of the weighted cut function. The question naturally arises: Can the constant $2$ be improved for general graphs? In this section, we answer this question in the negative, implying optimality of the decomposition in terms of the constant $c$. We start with a general observation on monotonic sum-decompositions of weighted cut functions. 

\begin{lem} \label{lem:clique} 
Let $G=(V,E)$ be a simple graph and let $H=(V,\cK)$ denote the hypergraph where $\cK$ is the family of vertex sets of complete subgraphs of $G$.
Let $w\colon E \to \bR_+$ be a weight function, and $d_w = \varphi_1 + \varphi_2$ be a monotonic sum-decomposition of the weighted cut function $d_w$. Then, there exist a weight function $w' \colon \cK \to \bR$ such that $\varphi_1 = i_{H, w'}$.
\end{lem}
\begin{proof}
We define $w'$ recursively for larger and larger complete subgraphs. If $w'$ is already defined on the proper subsets of $K \in \cK$, then we let 
\begin{equation} \label{eq:cliqueweight}
w'(K) \coloneqq \varphi_1(K) - \sum_{\substack{K' \in \cK \\ K' \subsetneq K}}w'(K').
\end{equation}
We show by induction on $|X|$ that $\varphi_1(X) = i_{H, w'}(X)$ holds for each $X \subseteq V$. If $X \in \cK$ then the statement holds by the choice of $w'$. Otherwise, $G[X]$ is not a complete subgraph of $G$, that is, there exist distinct vertices $u, v \in X$ such that $uv \not \in E$. Using \cref{lem:modular1} and equation \eqref{eq:i}, it follows that both $\varphi_1$ and $i_{H, w'}$ are modular on the pair $X-u, X-v$. Since $\varphi_1$ and $i_{H, w'}$ are equal on each of the sets $X\setminus\{u,v\}$, $X-u$, and $X-v$, it follows that $\varphi_1(X) = i_{H, w'}(X)$ holds as well, concluding the proof of the lemma.
\end{proof}
We note that for $K=\{v_1,\dots, v_k\} \in \cK$, equation \eqref{eq:cliqueweight} yields $w'(K) = (-1)^{k} \cdot V_{\varphi_1}(\emptyset; \{v_1\}, \dots, \{v_k\})$.  As $\varphi_1$ is increasing and submodular, this implies that $w'(K) \ge 0$ for $k=1$ and $w'(K) \le 0$ for $k=2$.

\begin{lem}\label{lem:9}
Let $G=(V,E)$ be a simple graph and let $H=(V,\cK)$ denote the hypergraph where $\cK$ is the family of vertex sets of complete subgraphs of $G$.
Let $w\colon E \to \bR_+$ be a weight function, and $d_w = \varphi_1 + \varphi_2$ be a monotonic sum-decomposition of the weighted cut function $d_w$. Furthermore, let $w'\colon \cK \to \bR$ be such that $\varphi_1 = i_{H, w'}$. Then \[d_w(v) \le w'(v) + \sum_{\substack{K \in \cK \\ v \in K}} w'(K)\]
for each $v\in V$.
\end{lem}
\begin{proof}
As $\varphi_1$ is increasing, we have
\[0 \le \varphi_1(V)-\varphi_1(V-v) = \sum_{\substack{K \in \cK \\ v \in K}} w'(K)\]
for any vertex $v \in V$. As $\varphi_2$ is decreasing, we have
\[0 \le -\varphi_2(v) = \varphi_1(v) - d_w(v)  = w'(v) - d_w(v).\]
By adding the two inequalities, the lemma follows.
\end{proof}

Our main result is as follows. 

\begin{thm} \label{thm:triangle-free}
Let $G=(V,E)$ be a simple triangle-free graph and $w\colon E \to \bR_+$ be a weight function. Then $\|d_w\|_+ = w(E)$.
\end{thm}
\begin{proof}
Recall that $\|d_w\|_+\le w(E)$ holds as $d_w = e_w + (-i_w)$ is a monotonic sum-decomposition. To prove $\|d_w\|_+ \ge w(E)$, let $d_w = \varphi_1+\varphi_2$ be a monotonic sum-decomposition. Let $H=(V,\cK)$ denote the hypergraph where $\cK$ consists of the vertex sets of complete subgraphs of $G$. Note that as $G$ is triangle-free, $\cK = \{\{v\}\mid v \in V\}\cup \{\{u,v\}\mid uv \in E\}$.
By \cref{lem:clique}, there exists a weight function $w'\colon \cK \to \bR$ such that $\varphi_1 = i_{H, w'}$. By Lemma~\ref{lem:9}, we have
\begin{equation} \label{eq:vineq}
d_w(v) \le w'(v) + \sum_{\substack{K \in \cK \\ v \in K}} w'(K)
\end{equation}
for each $v \in V$. Summing this inequality for $v \in V$ and using that $G$ is triangle-free, we get 
\[2 w(E) = \sum_{v \in V} d_w(v) \le 2 i_{H, w'}(V),\]
concluding the proof of the theorem.
\end{proof}

By Proposition~\ref{prop:erdos}, for any $c < 2$ there exists a simple triangle-free graph $G=(V,E)$ such that $\max_{X \subseteq V} d(X) > c \cdot |E|$. Therefore, \cref{thm:triangle-free} implies the following.

\begin{cor} For any $c<2$, there exists a graph whose cut function does not admit a $c$-bounded monotonic sum-decomposition.
\end{cor}

%%%%%%%%%%%%%%
\subsection{An upper bound on \texorpdfstring{$\|d_w\|_+$}{|dw|+}}
\label{sec:triangle}
%%%%%%%%%%%%%%

As a counterpart of \cref{thm:triangle-free}, we show how to improve the sum-decomposition $d_w = e_w - i_w$ if triangles are present. First we discuss complete graphs.

\begin{lem} \label{lem:complete}
The cut function of a complete graph admits a $1$-bounded monotonic sum-de\-com\-po\-si\-tion. 
\end{lem}
\begin{proof}
Observe that the cut function of a complete graph on $n$ vertices can be written as $d(X)=h(|X|)$, where $h(k)=k\cdot(n-k)$ is concave on $\bZ_+$. A similar idea as in Section~\ref{sec:prelim} shows that such functions admit 1-bounded monotonic sum-decompositions.
\end{proof}

Using \cref{lem:complete}, we obtain the following bound for general graphs.

\begin{thm} \label{thm:upper}
Let $G=(V,E)$ be a simple graph, $w\colon E \to \bR_+$ be a weight function, and let $\cK$ denote the vertex sets of complete subgraphs of $G$. 
Then, 
\[\|d_w\|_+ \le \min\Bigl\{\sum_{K \in \cK} \bigl\lceil \tfrac{|K|}{2} \bigr\rceil \cdot \bigl\lfloor \tfrac{|K|}{2}\bigr\rfloor \cdot z(K)\, \bigl\vert\, z\colon \cK \to \bR_+, \sum_{\substack{K \in \cK\\ e \in E[K]}} z(K) = w(e) \text{ for each $e \in E$} \Bigr\}.\]
\end{thm}
\begin{proof}
Let $z\colon \cK \to \bR_+$ be a function such that $\sum [z(K) \mid K \in \cK,\ e \in E(K)] = w(e)$ holds for each edge $e \in E$. For $K \in \cK$, let $d_K$ denote the (unweighted) cut function of the graph $G[K]$, and for a function $f\colon 2^K \to \bR_+$, let $f'\colon 2^V\to \bR_+$ denote the function defined by $f'(X) = f(X \cap K)$ for $X \subseteq V$. Then, $d_w = \sum_{K \in \cK} z(K) \cdot d'_K$ holds.
For each $K \in \cK$, \cref{lem:complete} implies that $d_K$ admits a decomposition $d_K = \varphi_{K, 1} + \varphi_{K,2}$ such that $\varphi_{K,1} \colon 2^K \to \bR$ is increasing submodular, $\varphi_{K,2}\colon 2^K \to \bR$ is decreasing submodular, and $\varphi_{K,1}(K) = \|d_K\| = \lceil |K|/2 \rceil \cdot \lfloor |K|/2\rfloor|$. By setting $\varphi_i \coloneqq \sum_{K \in \cK} z(K) \cdot \varphi'_{K,i}$ for $i \in [2]$, we get a monotonic sum-decomposition $d_w = \varphi_1 + \varphi_2$ such that $\varphi_1(V) = \sum_{K \in \cK} z(K) \cdot \varphi_{K,1}(K) = \sum_{K \in \cK} z(K) \cdot \lceil |K|/2 \rceil \cdot \lfloor |K|/2\rfloor$.
\end{proof}

We note that the upper bound given by \cref{thm:upper} is not tight: it can be shown that if $G$ is obtained by deleting one edge of the complete graph $K_7$ and $w$ is chosen to be identically one, then $\|d_w\| = 12$ while the upper bound given by the theorem is $12.5$. 

The bound given by \cref{thm:upper} is related to triangle packings.
Let $G=(V,E)$ be a graph, $w\colon E \to \bR_+$ be a weight function, and let $\mathcal{T}$ denote the vertex sets of the triangles of $G$. A function $x\colon \mathcal{T} \to \bR_+$ is called a \emph{weighted fractional triangle packing} if $\sum[x(T) \mid T \in \mathcal{T}, e \in E[T]] \le w(e)$ holds for each $e \in E$.
If $w$ is integer-valued, then an integer-valued weighted fractional triangle packing is called a \emph{weighted triangle packing}. A function $y\colon E \to \bR_+$ is called a \emph{fractional triangle cover} if $\sum[y(e) \mid e \in [T]] \ge 1$ holds for each triangle $T\in\cT$, and an integer-valued fractional triangle cover is called a \emph{triangle cover}. 
For a function $x\colon \cT \to \bR_+$ we use the notation $x(\cT) \coloneqq \sum[x(T) \mid T \in \cT]$, and for a function $y\colon E \to \bR_+$ we use $y_w(E) \coloneqq \sum[w(e)\cdot y(e) \mid e \in E]$.
The \emph{weighted triangle packing number} $\nu_w(G)$, the \emph{weighted fractional triangle packing number} $\nu^*_w(G)$, the  \emph{weighted triangle cover number} $\tau_w(G)$, and the \emph{weighted fractional triangle cover number} $\tau^*_w(G)$ are
\begin{align*}
\nu_w(G) & \coloneqq \max\{x(\cT) \mid  x \text{ is a weighted triangle packing}\}, \\
\nu^*_w(G) & \coloneqq \max\{x(\cT) \mid  x \text{ is a weighted fractional triangle packing}\}, \\
\tau_w(G) & \coloneqq \min\{y_w(E) \mid  y \text{ is a triangle cover}\}, \\
\tau^*_w(G) & \coloneqq \min\{y_w(E) \mid  y \text{ is a fractional triangle cover}\}.
\end{align*}
Note that $\nu_w(G) \le \nu^*_w(G) = \tau^*_w(G) \le \tau_w(G)$ holds for every graph $G$ by linear programming duality. We dismiss the subscript $w$ if the weight function is identically $1$. It is easy to see that $\tau(G) \le 3\nu(G)$ holds for any simple graph $G$, while Tuza's \cite{tuza1981conjecture} famous conjecture asserts that the stronger inequality $\tau(G) \le 2\nu(G)$ also holds. Chapuy, DeVos, McDonald, Mohar, and Scheide~\cite{chapuy2014packing} extended Tuza's conjecture to the weighted setting by asserting that $\tau_w(G) \le 2\nu_w(G)$ holds for any weight function $w\colon E\to \bZ_+$. By extending results of Haxell~\cite{haxell1999packing} and Krivelevich~\cite{krivelevich1995conjecture} for the unweighted case, they showed that the relaxations $\tau_w(G) \le 66/23 \cdot \nu_w(G)$, $\tau_w^*(G) \le 2\nu_w(G)$, and $\tau_w(G) \le 2\nu_w^*(G)$ hold. \cref{thm:upper} implies the following.

\begin{cor} \label{cor:nu*}
Let $G=(V,E)$ be a simple graph and $w\colon E\to\bR_+$ be a weight function. Then, $\|d_w\|_+\le w(E) - \nu^*_w(G)$.
\end{cor}
\begin{proof}
Let $\cT$ and $\cK$ denote the vertex sets of triangles and complete subgraphs of $G$, respectively. 
Let $x\colon \cT \to \bR_+$ be a weighted fractional packing such that $x(\cT) = \nu^*_w(G)$. Define $z\colon \cK \to \bR$ by  
\[z(K) = \begin{cases} x(K) & \text{if $|K|=3$,} \\
w(uv) - \sum[x(T) \mid T \in \cT, e \in E[T]] & \text{if $K=\{u,v\}$ for an edge $uv\in E$,} \\
0 & \text{otherwise.}
\end{cases}
\]
Then, $\sum[z(K) \mid K \in \cK,\ e \in E[K]] = w(e)$ holds for each $e \in E$ by the definition of $z$, $z \ge 0$ as $x$ is a weighted fractional triangle packing, and 
\begin{align*}
\sum_{K \in \cK} \bigl\lceil \tfrac{|K|}{2} \bigr\rceil \cdot \bigl\lfloor \tfrac{|K|}{2}\bigr\rfloor \cdot z(K) =  \sum_{uv \in E} z(\{u,v\}) + \sum_{T \in \cT} 2\cdot z(T) = w(E) - \sum_{T \in \cT} x(T) = w(E)-\nu^*_w(G),
\end{align*}
concluding the proof.
\end{proof}

Note that the bound given by \cref{thm:upper} is stronger than the one given by \cref{cor:nu*}, e.g.\ for the complete graph $K_5$ with weight function identically one. 

\begin{ex}\label{ex:planar}
For triangulated planar graphs, \cref{cor:nu*} implies the existence of a $1$-bounded monotonic sum-decomposition of the cut function $d$. To see this, let $G=(V,E)$ be a triangulated planar graph, and let $\cT$ denote the set of faces of $G$. Then, setting $x(T)=1/2$ for every $T\in\cT$ results in a fractional triangle packing of value $1/2\cdot|\cT|=|E|/3$, hence $|E|-\nu^*(G)=2/3\cdot|E|$. On the other hand, the maximum size of a cut in $G$ is exactly $2|E|/3$. Indeed, by cut-cycle duality, every cut of $G$ corresponds to an Eulerian subgraph of its dual $G^*=(V^*,E^*)$. Since $G^*$ is a cubic bridgeless graph, it contains a perfect matching $M$ by Petersen's theorem~\cite{petersen1891theorie}. Note that $M$ has size $|V^*|/2=|E^*|/3=|E|/3$. The complement of this perfect matching hence corresponds to a cut of $G$ of size $2\cdot|E|/3$.
\end{ex}

\begin{ex}
Using \cref{lem:clique} and \cref{cor:nu*}, we compute $\|d\|_+$ for the unweighted cut function $d$ of the wheel graph $W_n$, defined by vertex set $V=\{u,v_1,\ldots, v_{n-1}\}$ and edge set $E=\{uv_i \mid 1 \le i \le n-1\} \cup \{v_iv_{i+1} \mid 1 \le i \le n-1\}$ where indices are meant in a cyclic order, i.e., define $v_0 \coloneqq v_{n-1}$ and $v_n \coloneqq v_1$. We claim that if
$n \ge 5$ and 
$d$ denotes the cut function of the wheel graph $W_n$, then $\|d\|_+ = 3(n-1)/2$.
As a corollary, we get that $\|d\|_+$ is not an integer if $n \ge 6$ is even. 

The upper bound $\|d\|_+ \le 3(n-1)/2$ follows by applying \cref{cor:nu*}.
To prove the lower bound $\|d\|_+ \ge 3(n-1)/2$, let $H=(V,\cK)$ denote the hypergraph where $\cK$ consists of the vertex sets of complete subgraphs of $W_n$.
Note that $|K|\le 3$ for each $K\in \cK$ as $n \ge 5$.
Let $d=\varphi_1+\varphi_2$ be a monotonic sum-decomposition of the cut function $d$,  and $w'\colon \cK \to \bR$ be the weight function provided by \cref{lem:clique} such that $\varphi_1 = i_{H, w'}$. As $\varphi_1$ is increasing, \begin{align} \label{eq:wheel1} 
0\le \varphi_1(V-u)-\varphi_1(V-u-v_i) = w'(v_i) + w'(\{v_i,v_{i-1}\}) + w'(\{v_i,v_{i+1}\}) \end{align}
holds for $i \in[n-1]$. 
%where we define $v_0 \coloneqq v_{n-1}$ and $v_n \coloneqq v_1$. 
As $\varphi_2=d-\varphi_1$ is decreasing,
\begin{align} \label{eq:wheel2} 1 = d(\{u,v_i\})-d(u) \le\varphi_1(\{u,v_i\}) - \varphi_1(u) = w'(v_i) + w'(\{u,v_i\}). \end{align}
By Lemma~\ref{lem:9} with $v=v_i$, we get
\begin{equation} \label{eq:wheel3} 3 \le 2 w'(v_i) + \sum_{x \in \{u,v_{i-1},v_{i+1}\}}  w'(\{v_i,x\}) + \sum_{x \in \{v_{i-1}, v_{i+1}\}} w'(\{v_i,u,x\}).\end{equation}
Applying \cref{lem:9} with $v=u$ and multiplying by two, we get
\begin{equation} \label{eq:wheel4} 2(n-1) \le 4w'(u) + 2\cdot \sum_{i=1}^{n-1} \left(w'(uv_i) + w'(uv_iv_{i+1})\right). \end{equation}
Adding up the inequalities \eqref{eq:wheel1}, \eqref{eq:wheel2} and \eqref{eq:wheel3} for $i \in [n-1]$, and adding them to \eqref{eq:wheel4}, we get $6(n-1) \le 4 \cdot i_{H,w'}(V)$, proving our claim.
%\end{proof}
\end{ex}

%%%%%%%%%%%%%%%%%%%%%%%%%%%%%%%%
\section{Conclusions and open problems}
\label{sec:open}
%%%%%%%%%%%%%%%%%%%%%%%%%%%%%%%%

In this paper, we focused on problems related to decompositions of submodular set functions. We characterized the unique largest charge minorizing an increasing submodular set function  which, when subtracted from the original function, still gives an increasing set function. We then considered the problem whether a submodular set function over an infinite domain can be written as the sum of an increasing and a decreasing submodular function, or as the difference of two increasing submodular functions. As a fundamental difference to the finite case, we showed that such decompositions do not always exist. We introduced the notion of weakly infinite-alternating set functions, and verified the existence of the required decompositions for the members of this class. Finally, we studied weighted cut functions of graphs and provided various constructions and bounds for those. We close the paper by a list of open problems.
\medskip

1. \textbf{Monotonicity conjecture.} Recall that for a set function $\varphi$, we denoted by $\|\varphi\|_+$ 
the minimum value of $\varphi_1(J)$ over monotonic sum-decompositions of $\varphi$. Let $G=(V,E)$ be a graph and $w,w'\colon E\to\bR_+$ be weight functions satisfying $w'\leq w$. The following natural, seemingly simple problem remains open: \emph{Does $\|d_{G,w'}\|_{+}\leq \|d_{G,w}\|_+$ 
always hold?} The question is motivated by the following observation. Let $G$ be a simple graph and let $E'\subseteq E$ be a triangle cover. Then, for $w'(e)\coloneqq w(e)\cdot(1-\mathbbm{1}(e\in E'))$ the support of $w'$ is a triangle-free graph so $\|d_{G,w'}\|_+\geq w(E\setminus E')$ by Theorem~\ref{thm:triangle-free}. Therefore, the above inequality would imply a lower bound of $w(E)-\tau_w(G)\leq \|d_{G,w'}\|_+\leq \|d_{G,w}\|_+$. When complemented with \cref{cor:nu*}, this would lead to the bounds $|E|-\tau_w(G)\leq \|d_w\|_+\leq |E|-\nu^*_w(G)$ for the value of $\|d_w\|_+$. By the result of Krivelevich~\cite{krivelevich1995conjecture}, we know that $\tau_w(G)\leq 2\cdot \nu^*_w(G)$, meaning that the lower and upper bounds would not be far from each other. 
\medskip

2. \textbf{Planar graphs.} Example~\ref{ex:planar} shows that for a triangulated planar graph $G=(V,E)$, the value of $\|d_G\|_+$ is as small as possible in the sense that it is equal to the maximum size of a cut in $G$. Nevertheless, the proof heavily uses the fact that the trivial decomposition $d_G=e_G+(-i_G)$ can be significantly improved due to $G$ having a large fractional triangle packing number. The question remains: \emph{Can the value of $\|d_G\|_+$ be determined if $G$ is planar?}
\medskip

3. \textbf{Difference of two submodular functions.} 
In the finite case, any (not necessarily submodular) set function $f$ can be written as the difference of two submodular functions by Theorem~\ref{thm:finite_diff-decomp}. In the infinite case, we know that a decomposition into the difference of two increasing submodular set functions does not exist in the general, but we can ask the following. \emph{Let $\varphi$ be a set function with bounded variation i.e., it can be written as the difference of two increasing set functions. Can it be written as the difference of two bounded submodular set functions?}

\bigskip

\paragraph{Acknowledgement.} The authors are grateful to Miklós Abért, Márton Borbényi, Balázs Maga, Tamás Titkos, László Márton Tóth and Dániel Virosztek for helpful discussions. 

András Imolay was supported by the Rényi Doctoral Fellowship of the Rényi Institute. Tamás Schwarcz was supported by the \'{U}NKP-23-3 New National Excellence Program of the Ministry for Culture and Innovation from the source of the National Research, Development and Innovation Fund. This research has been implemented with the support provided by the Lend\"ulet Programme of the Hungarian Academy of Sciences -- grant number LP2021-1/2021 and by the Dynasnet European Research Council Synergy project -- grant number ERC-2018-SYG 810115.

%%%%%%%%%%%%%%%%%%%%%%%%%%%%%%%%
\bibliographystyle{abbrv}
\bibliography{decomposition}

\end{document}